\documentclass[a4paper,10pt]{article}

\usepackage[T1]{fontenc}
\usepackage[utf8]{inputenc}

\usepackage{latexsym}
\usepackage{mathrsfs}
\usepackage{amsfonts,amsmath,amssymb,amsthm}
\usepackage{bm}
\usepackage{stmaryrd}

\usepackage{indentfirst}
\usepackage{verbatim}
\usepackage{microtype}
\usepackage{float}
\usepackage{enumitem}

\usepackage{multirow}
\usepackage{diagbox}
\usepackage{booktabs}
\usepackage{array}
\usepackage{tabularx}
\usepackage[table]{xcolor}

\usepackage{graphicx}
\usepackage{epstopdf}
\usepackage{subcaption}
\usepackage{tikz}
\usepackage{pgfplots}
\usetikzlibrary{fit}
\pgfplotsset{compat=1.18}

\usepackage{pdfpages}
\usepackage{algorithm}
\usepackage{algorithmic}
\usepackage[title]{appendix}

\usepackage[unicode,colorlinks,linkcolor=blue,citecolor=red]{hyperref}

\numberwithin{equation}{section}
\allowdisplaybreaks

\newtheorem{theorem}{Theorem}[section]
\newtheorem{lemma}[theorem]{Lemma}
\newtheorem{remark}[theorem]{Remark}

\newcommand{\safeincludegraphics}[2][]{%
  \IfFileExists{#2}{\includegraphics[#1]{#2}}{%
    \fbox{\parbox[c][0.18\textheight][c]{0.78\linewidth}{\centering
    Placeholder for Figure\\[0.5ex]\texttt{\detokenize{#2}}}}%
  }%
}

\renewcommand{\epsilon}{\varepsilon}

\begin{document}

\begin{center}
{\Large\bfseries Adaptive Randomized Neural Networks for PDEs:
Structure-Aware Feature Redistribution and Residual-Driven Block Enrichment}
\end{center}

\begin{center}
{\large You Yang}\footnote{School of Mathematics and Statistics, Xi'an Jiaotong University, Xi'an, Shaanxi 710049, China. Email: {\tt yangyou@stu.xjtu.edu.cn}}
\quad {\rm and} \quad
{\large Fei Wang}\footnote{School of Mathematics and Statistics \& State Key Laboratory of Multiphase Flow in Power Engineering, Xi'an Jiaotong University, Xi'an, Shaanxi 710049, China. The work of this author was partially supported by the National Natural Science Foundation of China (Grant No.\ 92470115). Email: {\tt feiwang.xjtu@xjtu.edu.cn}}
\end{center}

\medskip

\begin{quote}
{\bf Abstract.} Randomized neural networks retain an output-linear least-squares solve, but globally sampled random features may fail to capture the spatial and scale structure of PDE solutions. We propose SFRBE--RaNN, a two-stage adaptive trial-space construction. SFR redistributes a fixed feature budget according to regional indicator masses and adapts regional weight scales through low-dimensional residual-based optimization. RBE constructs nonredundant enrichments by screening multiple activation families, removing represented response directions, and selecting the smallest admissible prefix. For solution-dependent operators, candidates are screened with a frozen operator and accepted only after nonlinear residual verification. For fixed linear problems, we establish an exact discrete residual-reduction identity and a conditional solution-error estimate. Building on this identity, we combine safeguarded SFR with RBE capture factors to derive a finite-horizon error bound involving baseline best approximation and empirical consistency, together with a conditional random-search success-probability bound. Numerical experiments on localized, multiscale, and nonlinear space--time problems demonstrate the method's effectiveness.
\end{quote}


\medskip


\section{Introduction}\label{sec:intro}

Classical numerical methods, including finite difference, finite element, discontinuous Galerkin, and spectral methods, remain central to PDE computation. Nevertheless, localized peaks, oscillations, layers, and evolving interfaces often require adaptive approximation spaces whose resolution follows the underlying solution structure.

Neural PDE solvers offer flexible trial spaces, including Deep Ritz, DGM, PINNs, and Galerkin neural methods~\cite{e2018deep,sirignano2018dgm,raissi2019physics,ainsworth2021galerkin}. Their performance, however, can depend heavily on architecture, sampling, and nonlinear training.

In contrast to fully trained neural PDE solvers, randomized neural networks (RaNNs) generate and freeze the hidden representation and determine only the output coefficients from the governing equations. For a shallow representation,
\begin{equation}\label{eq:intro_rann_representation}
    u_h(\bm x) = \sum_{j=1}^{m} c_j \rho\!\left( \bm w_j^\top\bm x+b_j \right),
\end{equation}
where the hidden parameters are fixed during each solve and the coefficients are obtained from a linear or linearized least-squares system. This linear output structure avoids nonlinear optimization of all network parameters and is compatible with a range of discretizations~\cite{dong2021local,chen2022bridging,shang2023randomized,shang2024randomized,sun2024local1,sun2024local2,li2025local,li2026structure,wang2026domain}.

The quality of a RaNN trial space depends on how its finite feature budget is configured: the distribution of hidden weights controls the characteristic length scales and directional variation of the features, feature anchors determine spatial placement, and the activation family determines the local response profile. A common scale model is
\begin{equation}\label{eq:intro_weight_sampling}
    \bm w_j\sim\mathcal U(-\bm r,\bm r)=\prod_{k=1}^{d}\mathcal U(-r_k,r_k),\qquad \bm r=(r_1,\ldots,r_d),
\end{equation}
where \(\bm r\) controls characteristic scales and directional variation. Feature anchors \(\bm a_j\in\Omega\) are introduced by setting \(b_j=-\bm w_j^\top\bm a_j\), so that the preactivation vanishes at \(\bm a_j\). Thus, a randomized trial space depends not only on its width, but also on its feature configuration and response span. The activation family controls the feature profile: \(\tanh\) produces smooth transitions, Gaussian features produce localized responses, and trigonometric features produce oscillatory variation.

Prior work has adapted several components of neural trial spaces. Global scales and low-dimensional feature distributions have been adjusted to control the overall randomized representation~\cite{dong2022Oncomputing,yang2026adaptivedistribution}; adaptive activation functions have been used to modify neural PDE responses~\cite{jagtap2020adaptive}; local trial spaces have been constructed using residual indicators~\cite{sun2024local2}; residual-guided neural basis construction has been used to progressively expand approximation spaces~\cite{dang2024adaptive,weng2026deep}; and randomized greedy methods have used randomly discretized dictionaries for weak-greedy enrichment~\cite{xu2025randomized}. Related residual-driven mechanisms in PINN methods guide sampling, weighting, or network updates~\cite{mcclenny2023self,wu2023comprehensive,gao2023failure}. These approaches operate at different levels of adaptation. Adaptive activation methods generally learn slopes or shape parameters together with the network parameters, whereas a strict RaNN solve keeps the hidden representation fixed and determines only the output coefficients from a linear or linearized least-squares system. Methods that adapt global feature distributions or local trial spaces do not, by themselves, specify how a fixed feature budget should be redistributed or how features aligned with missing residual directions should be added. This motivates a progressive two-stage approach to adaptive trial-space design: first adapting the existing trial space, and then expanding it with features better suited to the problem at hand.

We propose SFRBE--RaNN as a two-stage adaptive trial-space construction tailored to this setting. SFR reconfigures the existing feature budget using residual, derivative, geometric, or problem-specific indicators to identify structure and background regions, reposition feature anchors, and refresh regional scales without increasing the total width. RBE then adaptively constructs residual-informed feature blocks: candidate families are residualized against the current response space, ranked by a common capture score, and truncated at the smallest admissible nonredundant prefix. Fixed residual operators use block-incremental updates, whereas solution-dependent operators use frozen-operator screening followed by safeguarded reassembly.

The main contributions of this paper are as follows:
\begin{itemize}
    \item \textbf{Structure-aware feature redistribution.}
    We introduce SFR to identify localized or physics-defined structures, redistribute a fixed feature budget among the detected structure regions and the background according to their relative indicator masses, and adapt region-specific random weight scales through low-dimensional residual-based optimization.

    \item \textbf{Residual-informed feature-block construction.} RBE compares candidate blocks from different activation families, residualizes their responses against the current response space, ranks them by residual capture, and retains the smallest admissible nonredundant prefix.

    \item \textbf{Residual reduction and finite-horizon error control.}
    For frozen linear least-squares subproblems, we derive an exact residual-reduction identity for the residual-orthogonal RBE update and a conditional PDE-error estimate. For fixed linear problems, we further combine safeguarded SFR with the cumulative RBE capture factors to establish a finite-horizon error bound, together with a success-probability bound under a conditional coverage assumption for the complete random candidate search. We also provide a separate local one-step safeguard estimate for Picard or Newton updates.
\end{itemize}

The remainder is organized as follows. Sections~\ref{sec:rann}--\ref{sec:rbe} present the randomized residual formulation and develop the SFR and RBE modules. Section~\ref{sec:theory} gives the analysis; Section~\ref{sec:numerical} reports the experiments; and the conclusion and appendices summarize the findings and implementation details.

\section{Randomized Neural PDE Formulation and Adaptive Framework}
\label{sec:rann}

\subsection{Randomized neural residual formulation}
\label{sec:rann_formulation}

A shallow randomized neural network defines the trial space and the corresponding approximation as
\begin{equation}\label{eq:rann_representation}
    \mathcal V_m = \operatorname{span}\{\phi_1,\ldots,\phi_m\}, \qquad \phi_j(\bm x) = \rho(\bm w_j^\top\bm x+b_j), \qquad u_h(\bm x) = \sum_{j=1}^{m}c_j\phi_j(\bm x).
\end{equation}
The hidden parameters are fixed during each residual solve, while the output coefficients \(\bm c=(c_1,\ldots,c_m)^\top\) are determined by least squares.

Consider
\begin{equation}\label{eq:pde_operator_form}
    \begin{aligned}
        \mathcal Lu &= f, && \text{in }\Omega,\\ \mathcal Bu &= g, && \text{on }\partial\Omega,\\ \mathcal Iu &= h, && \text{on }\Omega_0.
    \end{aligned}
\end{equation}

Let $A_r$, $A_b$, and $A_0$ be the residual-response matrices evaluated at the interior, boundary, and initial collocation points, respectively. Let $\eta_b>0$ and $\eta_0>0$ denote the boundary and initial residual weights. The weighted residual system is defined as
\begin{equation}\label{eq:weighted_residual_system}
    A =
    \begin{bmatrix}
        A_r\\ \sqrt{\eta_b}A_b\\ \sqrt{\eta_0}A_0
    \end{bmatrix},
    \qquad \bm y =
    \begin{bmatrix}
        \bm f_r\\ \sqrt{\eta_b}\bm g_b\\ \sqrt{\eta_0}\bm h_0
    \end{bmatrix}.
\end{equation}
The output coefficients are computed from
\begin{equation}\label{eq:rann_lsq}
    \bm c = \arg\min_{\bm v} \|A\bm v-\bm y\|_2^2.
\end{equation}

For nonlinear PDEs, each outer iteration produces a linear or linearized least-squares problem of the form
\begin{equation}\label{eq:nonlinear_lsq}
    \bm c^{k+1} = \arg\min_{\bm v} \|A^k\bm v-\bm y^k\|_2^2.
\end{equation}
Here \(A^k\) and \(\bm y^k\) denote the matrix and right-hand side associated with the \(k\)-th linearized subproblem. For a fixed-operator Picard iteration, \(A^k\) remains fixed while \(\bm y^k\) is updated; for a Newton-type iteration, \(A^k\) and \(\bm y^k\) are assembled from the current linearization.

\subsection{Overall adaptive randomized neural network framework}
\label{sec:overall_framework}

SFRBE--RaNN updates the randomized trial space through two complementary modules:
\[ \mathcal V_0 \xrightarrow{\mathrm{SFR}} \widetilde{\mathcal V}_0 \xrightarrow{\mathrm{RBE}} \mathcal V_1 \xrightarrow{\mathrm{RBE}} \mathcal V_2 \xrightarrow{\mathrm{RBE}} \cdots . \]

SFR reconfigures the existing fixed width space through feature anchors, neuron allocations, biases, and regional hidden-weight scales, without changing the total number of raw features. RBE then adds a nonredundant space aligned with the unresolved residual:
\[ \mathcal V_{L+1} = \mathcal V_L + \mathcal W_L^{\rm new}. \]
Here \(\mathcal W_L^{\rm new}\) is the accepted enrichment space, represented by accepted composite features for fixed residual operators and by accepted raw features when the operator is updated. Here \(m_{\rm full}\) and \(m_{\rm scr}\) denote the full candidate-block size and the screening size, respectively. Figure~\ref{fig:sfrbe_modules} summarizes this workflow: the initial randomized space is diagnosed and refreshed by SFR, and the remaining residual is enriched by RBE.

\begin{figure}[t]
    \centering
    \includegraphics[width=\linewidth]{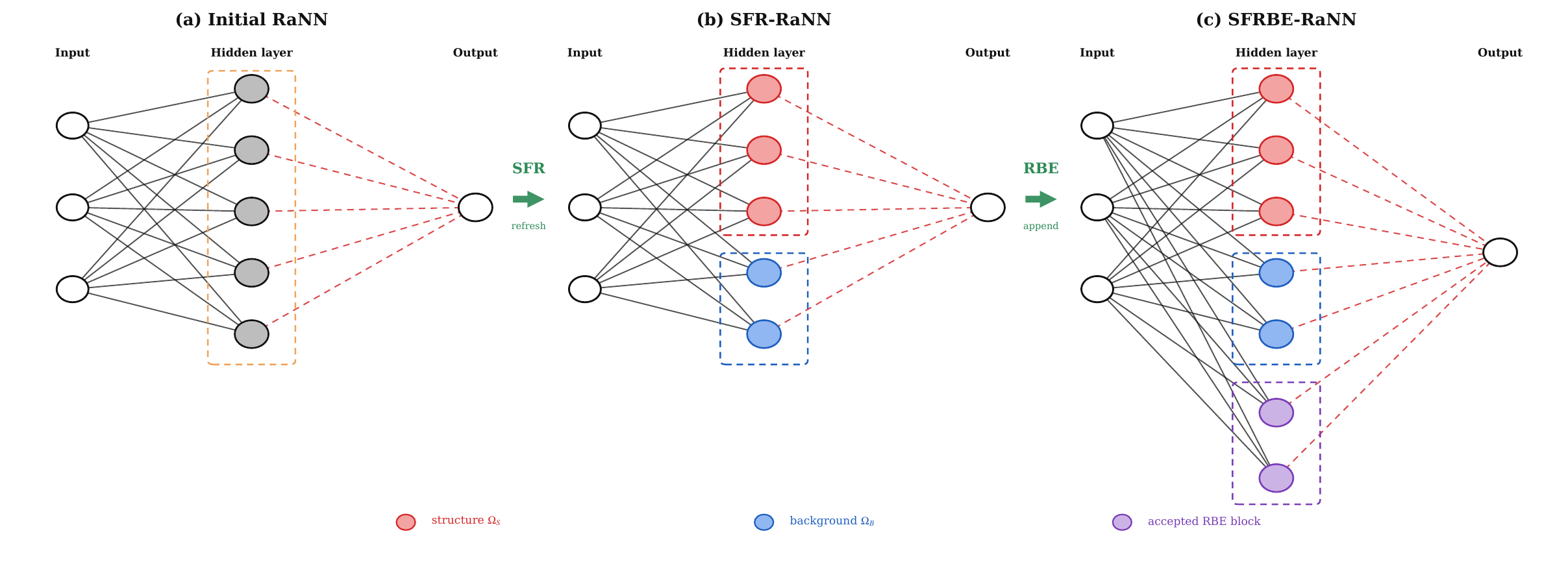}
    \caption{Standard randomized neural network followed by the two adaptive modules. SFR diagnoses the residual structure and refreshes the existing feature budget through regional anchors and scales. RBE screens candidate activation families, removes represented directions, and appends only an accepted residual-aligned block.}
    \label{fig:sfrbe_modules}
\end{figure}

\section{Structure-Aware Feature Redistribution}
\label{sec:sfr}

Let $u_h$ denote the current randomized neural-network approximation and let $\mathcal T_d=\{K_i\}_{i\in\mathcal I_d}$ be a Cartesian diagnostic grid, which is used for marking. SFR keeps the total number of features fixed and relocates their anchors only when the diagnostic indicator is spatially concentrated.

\subsection{Indicator and global marking}
\label{sec:sfr_detection}

For two-dimensional diagnostic grids, each valid cell $K_i$ with effective measure $\mu_i$ is evaluated at four tensor-product Gauss--Legendre points $z_{iq}$, $q=1,\ldots,4$. Let $Q_i=\{q\in\{1,\ldots,4\}:z_{iq}\in\Omega\}$ and $\mathcal I_d^{\rm v}=\{i\in\mathcal I_d:Q_i\ne\varnothing\}$; cells outside $\mathcal I_d^{\rm v}$ are excluded from marking and sampling. For $i\in\mathcal I_d^{\rm v}$, set $\omega_{iq}=\mu_i/|Q_i|$ and define
\begin{equation}
    r_i=\frac{1}{\mu_i}\sum_{j=1}^{J}\lambda_j \sum_{q\in Q_i}\omega_{iq} \left|\mathcal R_j(u_h;z_{iq})\right|, \qquad g_i=\frac{1}{\mu_i}\sum_{j=1}^{J}\lambda_j \sum_{q\in Q_i}\omega_{iq} \left\|\nabla u_{h,j}(z_{iq})\right\| . \label{eq:sfr_local_indicator}
\end{equation}
Here \(\mathcal R_j(u_h;\cdot)\) denotes the residual of the \(j\)th equation, including any prescribed row weighting. Also, $J=1$ for the scalar problems considered in the experiments and \(\lambda_j=1\). For space--time problems, the diagnostic cells and their effective measures include all independent coordinates. In the three-dimensional space--time Burgers run, the indicators are evaluated at the center of each Cartesian cell with measure \(\mu_{ijk}=h_xh_yh_t\). With a small numerical constant \(\varepsilon_n\), the normalized score is
\begin{equation}
    \widehat r_i=\frac{r_i}{\max_{k\in\mathcal I_d^{\rm v}} r_k+\varepsilon_n}, \qquad \widehat g_i=\frac{g_i}{\max_{k\in\mathcal I_d^{\rm v}} g_k+\varepsilon_n}, \qquad s_i=\frac{\widehat r_i+\widehat g_i}{\max_{k\in\mathcal I_d^{\rm v}}(\widehat r_k+\widehat g_k)+2\varepsilon_n}. \label{eq:sfr_score}
\end{equation}
The area-aware cell mass is $\eta_i=\mu_i s_i$.

The marked set is one global D{\"o}rfler set,
\[ \mathcal M\in\operatorname*{arg\,min}_{\mathcal A\subseteq \mathcal I_d^{\mathrm v}}\left\{|\mathcal A|:\sum_{i\in\mathcal A}\eta_i\ge\vartheta\sum_{i\in\mathcal I_d^{\mathrm v}}\eta_i\right\}\]
For two-dimensional diagnostic grids, the marked cells are labeled by eight-neighbour connectivity, producing candidate components $\widetilde{\mathcal C}_\ell$, $\ell=1,\ldots,K_{\rm raw}$. A candidate component is retained when
\begin{equation}
    \frac{\sum_{i\in\widetilde{\mathcal C}_\ell}\eta_i} {\sum_{i\in\mathcal I_d^{\rm v}}\eta_i+\varepsilon_n} \ge \tau_{\rm conn} . \label{eq:sfr_component}
\end{equation}
After filtering, reindex the survivors as \(\mathcal C_\ell\), \(\ell=1,\ldots,K_S\), set \(\mathcal I_S=\bigcup_{\ell=1}^{K_S}\mathcal C_\ell\) and \(\mathcal I_B=\mathcal I_d^{\rm v}\setminus\mathcal I_S\), and use \(\ell\) for a structure component and $B$ for the complementary background. The connectivity and indicator mass filter removes isolated cells with low scores while preserving small components with substantial accumulated residual and gradient scores.

\subsection{Anchor redistribution and regional scales}
\label{sec:sfr_redistribution}

For a ridge feature,
\begin{equation}
    \phi_j(\bm x)=\rho(\bm w_j^{\mathsf T}\bm x+b_j), \qquad b_j=-\bm w_j^{\mathsf T}\bm a_j, \label{eq:sfr_anchor}
\end{equation}
the anchor \(\bm a_j\) controls the feature location, whereas \(\bm w_j\) controls its direction and scale. To retain a nonzero sampling probability in low-score cells, SFR defines
\begin{equation}
    \bar\eta= \frac{\sum_{i\in\mathcal I_d^{\rm v}}\eta_i} {\sum_{i\in\mathcal I_d^{\rm v}}\mu_i}, \qquad d_i=\eta_i+\delta_f\bar\eta\mu_i, \qquad \delta_f=10^{-3}. \label{eq:sfr_density}
\end{equation}
Let \(m\) be the fixed total feature budget, let \(\mathcal C=\bigcup_{\ell=1}^{K_S}\mathcal C_\ell\) denote the union of the retained structure components, and let \(\mathcal I_B=\mathcal I_d^{\rm v}\setminus\mathcal C\) denote the complementary background. The structure--background allocation is determined by the normalized density mass
\begin{equation}
    \alpha_S= \frac{\sum_{i\in\mathcal C}d_i} {\sum_{i\in\mathcal I_d^{\rm v}}d_i}, \qquad m_S=\operatorname{round}(\alpha_Sm), \qquad m_B=m-m_S. \label{eq:sfr_budget_split}
\end{equation}
Thus \(m_S\) features are assigned to the retained structure region and \(m_B\) features to the background, so that the total width remains fixed. If several structure components are retained, the structure budget \(m_S\) is distributed among them in proportion to their density masses. Conditional on the selected region, cells are sampled according to
\begin{equation}
    \mathbb P(\bm a_j\text{ selects }K_i\mid\ell)=\frac{d_i}{\sum_{k\in\mathcal C_\ell}d_k},\qquad \mathbb P(\bm a_j\text{ selects }K_i\mid B)=\frac{d_i}{\sum_{k\in\mathcal I_B}d_k}. \label{eq:sfr_sampling}
\end{equation}
After a cell is selected, each anchor is sampled uniformly from the physical extent of that cell. Each retained component, as well as the background, is assigned its own adaptively selected scale parameters; admissible scale-selection mechanisms are described in Appendix~\ref{app:rbe_parameter_refinement}. After the regional features are built, all output coefficients are recomputed by the same weighted least-squares procedure as in the global solve.

\subsection{Refresh gate and safeguard}
\label{sec:sfr_refresh}

Let $N_{\rm valid}=|\mathcal I_d^{\rm v}|$ denote the number of valid diagnostic cells. When $K_S\ge1$, let $m_\ell$ be the feature count allocated to component $\mathcal C_\ell$, with $\sum_{\ell=1}^{K_S}m_\ell=m_S$, and let $\widehat\alpha_\ell$ be the fraction of current anchors in that component. Set $\alpha_\ell^{\rm tar}=m_\ell/m$. The effective indicator support and the coverage deficit (when $K_S\ge1$) are defined by
\begin{equation}
    \chi_{\rm eff}= \frac{\left(\sum_{i\in\mathcal I_d^{\rm v}}\eta_i\right)^2} {N_{\rm valid}\sum_{i\in\mathcal I_d^{\rm v}}\eta_i^2+\varepsilon_n}, \quad D_{\rm cov}=\max_\ell \frac{[\alpha_\ell^{\rm tar}-\widehat\alpha_\ell]_+} {\alpha_\ell^{\rm tar}+\varepsilon_n}. \label{eq:sfr_support_coverage}
\end{equation}
SFR is refreshed only if the residual--gradient indicator is sufficiently localized, at least one component survives the mass filter, and the current anchors do not cover the target components:
\begin{equation}
    \chi_{\rm eff}\le\tau_{\rm eff},\qquad K_S\ge1,\qquad D_{\rm cov}\ge\tau_{\rm cover}. \label{eq:sfr_gate}
\end{equation}

The redistributed basis is compared with the old basis on the same collocation set. Let $E_{\rm acc}^{\rm old}$ and $E_{\rm acc}^{\rm new}$ denote the corresponding weighted $\ell^2$ residual norms. The refresh is accepted only when
\begin{equation}
    E_{\rm acc}^{\rm new}\le (1-\eta_{\rm acc})E_{\rm acc}^{\rm old}, \qquad \eta_{\rm acc}=0 \label{eq:sfr_accept}
\end{equation}
for the default non-increasing safeguard. For nonlinear problems, the acceptance residual is recomputed using the full nonlinear operator after the configured refinement step. If the acceptance criterion is not satisfied, the old basis and coefficients are restored. Thus, for both linear and nonlinear problems, the safeguard ensures that the post-SFR weighted empirical residual does not exceed its pre-SFR value. For fixed linear problems, this property is also used to establish the finite-horizon error bound in Theorem~\ref{thm:finite_horizon_sfrbe}.

\section{Residual-Driven Block Enrichment}
\label{sec:rbe}

RBE enlarges the trial space when the current residual still contains unresolved directions. At stage $L$, let $A_L$ denote the weighted residual-response matrix of the current features, assembled according to \eqref{eq:weighted_residual_system}. The corresponding least-squares system is
\begin{equation}
    A_L\bm c_L\approx\bm y. \label{eq:rbe_system}
\end{equation}

\subsection{Residual-weighted candidates and screening}
\label{sec:rbe_candidates}

Let $\mathcal F$ denote the problem-specific set of candidate families. In the implementation considered here, we use an optional family scheduling rule to diversify successive RBE enrichments and avoid repeatedly adding blocks from the same activation family. After a block from family $F$ is accepted, $F$ is excluded from the next candidate construction and screening step. All families are available at the first RBE step, and $F$ becomes available again once a block from another family has been accepted. RBE candidate centers are sampled from the current discrete interior PDE point set \(X_{\rm int}=\{\boldsymbol z_i\}_{i=1}^{N_{\rm int}}\), using the residual-weighted probabilities below. With $r_i$ denoting the $i$th interior entry of $\bm r_L$, the implementation uses
\[ \widehat r_i= \frac{|r_i|} {\displaystyle\max_{1\le k\le N_{\rm int}}|r_k|+\varepsilon_n}, \qquad s_i=\delta+\widehat r_i^{\,q}, \qquad p_i=(1-\alpha_{\rm mix})\frac{1}{N_{\rm int}}+\alpha_{\rm mix}\frac{s_i}{\sum_{k=1}^{N_{\rm int}}s_k}, \label{eq:rbe_sampling} \]
where the standard values are $\delta=3\times10^{-3}$, $q=2$, and $\alpha_{\rm mix}=0.85$. Hence every point retains a nonzero probability while large residuals are preferentially sampled. A screening block contains \(m_{\rm scr}\) features. Its low-dimensional family parameters (for example, a scale, frequency, or shape parameter) may be refined by the optional SFRBE parameter search described in Appendix~\ref{app:rbe_parameter_refinement}. The selected parameters are frozen when the full candidate block is formed.

The candidate families and their typical structures are summarized in Table~\ref{tab:rbe_activation_families}.

\begin{table}[!ht]
    \centering
    \small
    \setlength{\tabcolsep}{4pt}
    \renewcommand{\arraystretch}{1.08}
    \caption{Candidate activation families screened by RBE. The family is chosen from the residual-capture score after residualization; no single activation is prescribed for all PDEs.}
    \label{tab:rbe_activation_families}
    \begin{tabularx}{\linewidth}{@{}l>{\raggedright\arraybackslash}p{0.40\linewidth}>{\raggedright\arraybackslash}X@{}}
        \toprule Family & Representative feature & Residual structure typically captured \\
        \midrule \texttt{tanh} & $\tanh(\bm w^{\mathsf T}\bm x+b)$ & smooth transitions and anisotropic layers \\
        \texttt{ridge\_gauss} & $\exp[-(\bm w^{\mathsf T}\bm x+b)^2]$ & localized ridge-like defects \\
        \texttt{sin} & $\sin(\pi(\bm w^{\mathsf T}\bm x+b))$ & oscillatory or high-frequency residuals \\
        \texttt{radial\_gauss} & $\exp[-\sum_\ell r_\ell^2(x_\ell-a_\ell)^2]$ & isolated peaks and compact defects \\
        \texttt{mexican\_hat} & $(d-s)\exp(-s/2),\quad s=\sum_{\ell=1}^{d}r_\ell^2(x_\ell-a_\ell)^2$ & sign-changing localized residuals \\
        \bottomrule
    \end{tabularx}
\end{table}

Let $Q_L$ be an orthonormal basis for the current retained response space and let $\bm r_L$ be the current weighted residual. For any candidate response matrix $C_F$, form
\begin{equation}
    C_{\perp,F}=(I-Q_LQ_L^{\mathsf T})C_F =U_F\Sigma_FV_F^{\mathsf T}+E_F, \label{eq:rbe_residualization}
\end{equation}
where \(U_F\Sigma_FV_F^{\mathsf T}\) contains the singular components retained by the configured relative SVD threshold and \(E_F\) contains the discarded components. Define the regularized capture and map-amplification measures by
\begin{equation}
    s_F=\frac{\|U_F^{\mathsf T}\bm r_L\|_2^2}{\|\bm r_L\|_2^2+\varepsilon_n},\qquad \kappa_F=\|V_F\Sigma_F^{-1}\|_2. \label{eq:rbe_screen_score}
\end{equation}
For the initial family screening, \eqref{eq:rbe_screen_score} is applied to the $m_{\rm scr}$-feature screening block $C_F^{\rm scr}$, and the resulting quantities are denoted by $s_F^{\rm scr}$ and $\kappa_F^{\rm scr}$. The threshold $\tau_{\rm cap}=10^{-4}$ is the minimum screening capture score required for a candidate family to proceed to full-block construction.

For the exact identities below, define the corresponding unregularized score for any retained SVD-truncated block by
\begin{equation}
    s_F^{(0)}=\frac{\|U_F^{\mathsf T}\bm r_L\|_2^2}{\|\bm r_L\|_2^2},\qquad \bm r_L\ne\bm 0. \label{eq:rbe_exact_score}
\end{equation}
Thus, $s_F$ and $s_F^{(0)}$ use the same retained SVD basis and differ only in the denominator regularization.

\subsection{Prefix selection and safeguarded enrichment}
\label{sec:rbe_update}

For a family passing the initial screening, RBE generates a full block of $m_{\rm full}$ features and ranks its columns by one-column residual capture, computed by applying \eqref{eq:rbe_screen_score} to each residualized column individually. It then tests the prescribed prefix sizes $p\in\mathcal P$. Let $C_{\perp,F,p}$ denote the residualized matrix formed by the first $p$ ranked features, and let $U_{F,p}$ denote the retained left singular vectors after SVD truncation. For each prefix, \(s_{F,p}\) and \(\kappa_{F,p}\) are defined by applying \eqref{eq:rbe_screen_score} to the retained SVD of \(C_{\perp,F,p}\). Likewise, \(s_{F,\mathrm{full}}\) and \(\kappa_{F,\mathrm{full}}\) denote the corresponding quantities obtained from the full block. The smallest admissible prefix is the smallest $p\in\mathcal P$ satisfying
\begin{equation}
    s_{F,p}\ge\theta_{\rm pref}s_{F,\rm full};\quad \kappa_{F,p}\le\kappa_{\max}^{\rm map};\quad \dim\operatorname{Range}(U_{F,p})\ge1;\quad s_{F,p}\ge\tau_{\rm stop}. \label{eq:rbe_prefix}
\end{equation}
Families are examined in decreasing order of $s_F^{\rm scr}$, and the first family with such a prefix is selected. The map-conditioning cap is denoted by \(\kappa_{\max}^{\rm map}\); it is $10^6$ for all cases. If no prefix satisfies these conditions, the candidate block is not added.

For a fixed linear operator, let $C_{\perp,*}=(I-Q_LQ_L^{\mathsf T})C_*$ be the residualized accepted block, and let $U_*$ denote the orthonormal basis obtained from its retained SVD components. By construction, $Q_L^{\mathsf T}U_*=0$ up to numerical roundoff. The least-squares correction contributed by the new block is therefore the projection of the current residual onto $\operatorname{Range}(U_*)$:
\begin{equation}
    \bm\gamma_*=U_*^{\mathsf T}\bm r_L,\qquad \bm r_{L+1}=\bm r_L-U_*\bm\gamma_*,\qquad Q_{L+1}=[Q_L,U_*]. \label{eq:rbe_linear_update}
\end{equation}

At outer iteration $k$, the solution-dependent operator is frozen for candidate screening. Once a block is selected, its raw feature columns are appended, the frozen augmented problem is reassembled, and the coefficients are recomputed:
\begin{equation}
    A_{L+1}^{(k)}=[A_L^{(k)},C_*^{(k)}],\qquad \bm c_{L+1}^{(k+1)}=\operatorname*{arg\,min}_{\bm v}\|A_{L+1}^{(k)}\bm v-\bm y^{(k)}\|_2^2. \label{eq:rbe_nonlinear_update}
\end{equation}
The resulting trial state is then refined with the full nonlinear operator before the acceptance test. The block is accepted only if the recomputed nonlinear residual decreases,
\begin{equation}
    \|\mathcal R_{\rm nl}(u_{L+1})\|_2 <\|\mathcal R_{\rm nl}(u_L)\|_2, \label{eq:rbe_accept}
\end{equation}
Otherwise the trial state is rolled back and enrichment stops. This residual safeguard is the key distinction between candidate screening and an accepted RBE update.

\subsection{SFR--RBE workflow}
\label{sec:sfrbe_workflow}

The complete procedure is therefore:
\begin{enumerate}[leftmargin=2em,itemsep=1pt,topsep=2pt]
    \item generate the initial randomized features and solve for the output coefficients;
    \item evaluate the residual--gradient indicator and, if \eqref{eq:sfr_gate} holds, perform one safeguarded SFR refresh;
    \item construct residual-weighted candidate blocks for the configured families, refine their low-dimensional family parameters, and screen them by \eqref{eq:rbe_screen_score};
    \item select the smallest admissible prefix satisfying
    \eqref{eq:rbe_prefix}, append it, and re-solve the augmented problem;
    \item for fixed linear operators, accept the selected block and update by \eqref{eq:rbe_linear_update}; for solution-dependent operators, accept the trial block only if \eqref{eq:rbe_accept} holds; repeat the enrichment procedure.
\end{enumerate}
The same workflow is used throughout; the block update and final acceptance test depend on whether the residual operator is fixed or solution-dependent.

\section{Residual Reduction and Error Analysis}
\label{sec:theory}

This section analyzes residual reduction and error control for SFRBE--RaNN. For fixed linear problems, we establish the exact residual reduction produced by an accepted RBE block, connect the residual capture score with the PDE solution error, and derive a finite-horizon error bound for safeguarded SFRBE with a conditional success-probability bound for the randomized candidate search. For solution-dependent operators, we give a separate one-step residual estimate for frozen Picard and Newton updates.

The fixed-linear results assume one fixed affine residual map, including its right-hand side, one weighted collocation rule, and exact unregularized projection onto the retained orthonormal SVD response spaces. Ridge regularization and inexact solves are implementation choices. The finite-horizon probabilistic statement is conditional on candidate coverage and empirical consistency, while the nonlinear result is treated separately because the residual operator may change with the current iterate.

\subsection{Discrete residual reduction}

Consider the weighted least-squares problem over the current retained residual-response space, represented by an orthonormal basis \(Q_L\):
\begin{equation}\label{eq:theory_lsq}
    \bm d_L = \arg\min_{\bm d}\|Q_L\bm d-\bm y\|_2^2 = Q_L^\top\bm y, \qquad \bm r_L = \bm y-Q_L\bm d_L.
\end{equation}
For a candidate block \(C\), define
\[ C_\perp = (I-Q_LQ_L^\top)C, \]
and let
\[ C_\perp = U\Sigma V^\top+E_{\rm trunc} \]
denote its decomposition into retained and discarded singular components, with \(E_{\rm trunc}\) containing the discarded components. The increment is formed only from \(\operatorname{Range}(U)\), so the discarded directions are not part of the augmented response space. For the fixed-linear retained block, the equivalent transformed features can be written as \(Z=C V\Sigma^{-1}\); hence \((I-Q_LQ_L^\top)Z=U\) when the retained SVD relation is exact.

\begin{theorem}[Residual reduction identity for the block incremental least-squares update]
    \label{thm:rbe_residual_decrease}
    Assume that the current and augmented coefficient vectors are obtained by exact unregularized least-squares projection in the retained spaces. In exact arithmetic, let \(Q_L^\top Q_L=I\), \(Q_L^\top U=0\), \(U^\top U=I\), and \(\bm r_L=(I-Q_LQ_L^\top)\bm y\neq\bm 0\). Let \(\operatorname{Range}(U)\) denote the retained nonredundant residual-response subspace of the SVD-truncated candidate block. For the accepted block \(C\), define \(s_C^{(0)}=\|U^\top\bm r_L\|_2^2/\|\bm r_L\|_2^2\). After augmenting the current response space by \(\operatorname{Range}(U)\), the block incremental least-squares residual satisfies
    \begin{equation}\label{eq:rbe_residual_update_theory}
        \bm r_{L+1} = \bm r_L-UU^\top\bm r_L,
    \end{equation}
    and
    \begin{equation}\label{eq:rbe_residual_identity_theory}
        \begin{aligned}
            \|\bm r_{L+1}\|_2^2 &= \|\bm r_L\|_2^2 - \|U^\top\bm r_L\|_2^2 &= \left(1-s_C^{(0)}\right) \|\bm r_L\|_2^2.
        \end{aligned}
    \end{equation}
\end{theorem}
\begin{proof}
    By construction, \(Q_L^\top U=0\) and \(U^\top U=I\). Hence \([Q_L,U]\) is an orthonormal basis of the augmented response space, and its least-squares residual is
    \[
    \begin{aligned}
        \bm r_{L+1} &=\bm y-[Q_L,U][Q_L,U]^\top\bm y\\ &=(I-Q_LQ_L^\top)\bm y-UU^\top\bm y =\bm r_L-UU^\top\bm r_L,
    \end{aligned}
    \]
    where the last equality follows from \(U^\top Q_L=0\). Since \(UU^\top\) is an orthogonal projector,
    \[ \|\bm r_{L+1}\|_2^2 = \|\bm r_L\|_2^2-\|U^\top\bm r_L\|_2^2. \]
    The definition of \(s_C^{(0)}\) completes the proof.
\end{proof}

\begin{remark}[Implementation score]
    Implementation uses the regularized score \(s_F\) for screening and stopping, whereas the exact identity uses \(s_C^{(0)}\) on the same retained block. Solution-dependent trials are reassembled and accepted only when the original nonlinear residual decreases.
\end{remark}

\subsection{Conditional PDE solution-error estimate}

Let $V$ be a Hilbert solution space on which the differential operator and the relevant boundary and initial traces are well defined, and let $Y_r$, $Y_b$, and $Y_0$ be the corresponding residual spaces. For a fixed linear problem, define the continuous residual loss by
\begin{equation}\label{eq:continuous_residual_loss}
    \begin{aligned}
        \mathscr L(v) ={}& \|\mathcal Lv-f\|_{Y_r}^2 + \eta_b\|\mathcal Bv-g\|_{Y_b}^2 + \eta_0\|\mathcal Iv-h\|_{Y_0}^2.
    \end{aligned}
\end{equation}

Assume that there exist constants \(C_{\rm s},C_{\rm c}>0\) such that
\begin{equation}\label{eq:graph_norm_theory}
    \begin{aligned}
        C_{\rm s}\|v\|_V^2 \le{}& \|\mathcal Lv\|_{Y_r}^2 + \eta_b\|\mathcal Bv\|_{Y_b}^2 + \eta_0\|\mathcal Iv\|_{Y_0}^2 \le C_{\rm c}\|v\|_V^2, \qquad v\in V.
    \end{aligned}
\end{equation}
This graph-norm stability condition is problem-dependent, as in~\cite{dang2024adaptive}. For the two-peak Poisson benchmark in Section~\ref{sec:twopeak_poisson}, this framework yields an explicit \(L^2\)-error estimate; its proof is given in Appendix~\ref{app:poisson_residual_bound}. For general linear problems, the corresponding residual-to-solution control is expressed through the graph-norm condition above.

Let \(u^\ast\in V\) be the exact solution satisfying
\[ \mathcal Lu^\ast=f, \qquad \mathcal Bu^\ast=g, \qquad \mathcal Iu^\ast=h. \]
Since \(\mathcal L\), \(\mathcal B\), and \(\mathcal I\) are linear, applying \eqref{eq:graph_norm_theory} to \(v-u^\ast\) gives
\begin{equation}\label{eq:loss_error_theory}
    C_{\rm s}\|v-u^\ast\|_V^2 \le \mathscr L(v) \le C_{\rm c}\|v-u^\ast\|_V^2.
\end{equation}

Let \(\bm R_r(v)\), \(\bm R_b(v)\), and \(\bm R_0(v)\) collect the sampled interior, boundary, and initial residuals, respectively, and let \(\bm W_r=\operatorname{diag}(\omega_i^r)\), \(\bm W_b=\operatorname{diag}(\omega_i^b)\), and \(\bm W_0=\operatorname{diag}(\omega_i^0)\). The empirical residual loss is
\begin{equation}\label{eq:empirical_residual_loss}
    \widehat{\mathscr L}_N(v) =\|\bm W_r^{1/2}\bm R_r(v)\|_2^2 +\eta_b\|\bm W_b^{1/2}\bm R_b(v)\|_2^2 +\eta_0\|\bm W_0^{1/2}\bm R_0(v)\|_2^2 .
\end{equation}
The weights are therefore incorporated into the weighted least-squares matrix and right-hand-side vector. Consequently,
\[ \widehat{\mathscr L}_N(u_L) = \|Q_LQ_L^\top\bm y-\bm y\|_2^2 = \|\bm r_L\|_2^2. \]

For two consecutive RBE approximations computed in the retained composite trial spaces, define
\begin{equation}\label{eq:discrepancy_theory}
    \Delta_{N,L} = \max_{v\in\{u_L,u_{L+1}\}} \left| \mathscr L(v) - \widehat{\mathscr L}_N(v) \right|.
\end{equation}
Thus \(\Delta_{N,L}\) measures the empirical-to-continuous residual discrepancy at the two consecutive approximations \(u_L\) and \(u_{L+1}\).

\begin{theorem}[Conditional solution-error estimate for RBE]
    \label{thm:rbe_conditional_error}

    Suppose that the graph-norm stability condition \eqref{eq:graph_norm_theory} holds and that the two consecutive approximations are exact unregularized least-squares projections for the same fixed linear residual operator and empirical residual rule. Let \(s_C^{(0)}\) denote the unregularized capture score of the accepted SVD-truncated block. Then the main one-step estimate is
    \begin{equation}\label{eq:rbe_error_bound_theory}
        \begin{aligned}
            \|u_{L+1}-u^\ast\|_V^2 \le\;& \frac{(1-s_C^{(0)})C_{\rm c}}{C_{\rm s}} \|u_L-u^\ast\|_V^2 + \frac{2-s_C^{(0)}}{C_{\rm s}} \Delta_{N,L}.
        \end{aligned}
    \end{equation}
\end{theorem}

\begin{proof}
    By \eqref{eq:loss_error_theory} and \eqref{eq:discrepancy_theory},
    \[
    \begin{aligned}
        C_{\rm s}\|u_{L+1}-u^\ast\|_V^2 &\le \widehat{\mathscr L}_N(u_{L+1})+\Delta_{N,L} =(1-s_C^{(0)})\widehat{\mathscr L}_N(u_L)+\Delta_{N,L}\\ &\le(1-s_C^{(0)})\big(C_{\rm c}\|u_L-u^\ast\|_V^2 +\Delta_{N,L}\big)+\Delta_{N,L},
    \end{aligned}
    \]
    where the equality follows from \eqref{eq:rbe_residual_identity_theory}.
\end{proof}

\begin{remark}[Conditional error interpretation]
    Theorem~\ref{thm:rbe_conditional_error} separates the effect of residual projection from the empirical-to-continuous consistency term. When the two losses agree, the estimate reduces to a contraction criterion:
    \[ s_C^{(0)}>1-\frac{C_{\rm s}}{C_{\rm c}}. \]
\end{remark}

\subsection{Finite-horizon randomized enrichment and SFRBE error control}
\label{sec:finite_horizon}

The actual capture factors can be accumulated along accepted fixed-linear updates. Together with the SFR safeguard, this relates the final error to the pre-SFR baseline without imposing a uniform statistical bound over an unbounded trial space. Let $V_{\rm base}\subset V$ be the finite-dimensional pre-SFR trial space, $u_{\rm base}$ its exact empirical least-squares minimizer, and $u_0$ the accepted or restored post-SFR state. Define
\[ v_{\rm base}^\star\in\arg\min_{v\in V_{\rm base}} \|v-u^\ast\|_V^2,\qquad A_{\rm base}=\|v_{\rm base}^\star-u^\ast\|_V^2, \]
and, for $L$ accepted RBE blocks,
\[ \Pi_L=\prod_{\ell=0}^{L-1}(1-s_\ell^{(0)}),\qquad\Pi_0=1, \]
where $s_\ell^{(0)}$ is the unregularized capture score of the retained block from \eqref{eq:rbe_exact_score}.

\begin{theorem}[Finite-horizon RBE capture and safeguarded SFRBE error]
    \label{thm:finite_horizon_sfrbe}
    Fix $L\ge1$ and consider the same fixed affine residual map \emph{(including its right-hand side)}, weighted collocation rule, and exact unregularized projection onto retained response spaces throughout SFR and the $L$ accepted RBE updates. Assume the graph-norm equivalence \eqref{eq:graph_norm_theory} and the SFR safeguard \eqref{eq:sfr_accept}. Suppose, for this realized final state, that the \emph{two-point} empirical-consistency estimates hold:
    \begin{equation}\label{eq:finite_horizon_consistency}
        \begin{aligned}
            \big|\mathscr L(v_{\rm base}^\star)- \widehat{\mathscr L}_N(v_{\rm base}^\star)\big| \le\varepsilon_{\rm base},\qquad \big|\mathscr L(u_L)-\widehat{\mathscr L}_N(u_L)\big| \le\varepsilon_L,
        \end{aligned}
    \end{equation}
    where $\varepsilon_{\rm base},\varepsilon_L\ge0$. Then every such successful trajectory satisfies
    \begin{equation}\label{eq:finite_horizon_bound}
        \displaystyle \|u_L-u^\ast\|_V^2\le \frac{C_{\rm c}}{C_{\rm s}}\Pi_L A_{\rm base} +\frac{\Pi_L\varepsilon_{\rm base}+\varepsilon_L}{C_{\rm s}}.
    \end{equation}
    In particular, if the implementation score of every accepted block is at least $\tau_{\rm stop}\in(0,1)$, then $\Pi_L\le(1-\tau_{\rm stop})^L$.

    For the random-search statement, let $G_\ell$ be the event that the \emph{complete} random candidate-generation, SVD truncation, conditioning, and prefix-selection procedure at stage $\ell$ finds an admissible block passing the stopping threshold. If, conditional on every preceding successful history $\mathcal F_\ell$,
    \begin{equation}\label{eq:finite_horizon_success}
        \mathbb P(G_\ell\mid\mathcal F_\ell)\ge p_0>0, \qquad \ell=0,\ldots,L-1,
    \end{equation}
    then $\mathbb P(G_0\cap\cdots\cap G_{L-1})\ge p_0^L$. If \eqref{eq:finite_horizon_consistency} holds with $\varepsilon_{\rm base},\varepsilon_L\le\varepsilon_N$ on every such successful history, then, with probability at least $p_0^L$,
    \begin{equation}\label{eq:finite_horizon_geometric}
        \|u_L-u^\ast\|_V^2\le \frac{C_{\rm c}}{C_{\rm s}}(1-\tau_{\rm stop})^L A_{\rm base} +\frac{2\varepsilon_N}{C_{\rm s}}.
    \end{equation}
    More generally, suppose an event $H_L$ on the ambient probability space satisfies $\mathbb P(H_L)\ge1-\delta_N$ and enforces the same two-point bounds on every successful path in $H_L$. Then \eqref{eq:finite_horizon_geometric} holds with probability at least $\max\{0,p_0^L-\delta_N\}$.
\end{theorem}
\begin{proof}
    By the SFR safeguard and baseline empirical optimality, $\widehat{\mathscr L}_N(u_0) \le\widehat{\mathscr L}_N(v_{\rm base}^\star)$. Iterating \eqref{eq:rbe_residual_identity_theory} for the unchanged weighted affine residual map gives $\widehat{\mathscr L}_N(u_L)=\Pi_L\widehat{\mathscr L}_N(u_0)$. Consequently, \eqref{eq:loss_error_theory} and the two estimates in \eqref{eq:finite_horizon_consistency} imply
    \[
    \begin{aligned}
        C_{\rm s}\|u_L-u^\ast\|_V^2 &\le\widehat{\mathscr L}_N(u_L)+\varepsilon_L\\ &\le\Pi_L\big(\mathscr L(v_{\rm base}^\star) +\varepsilon_{\rm base}\big)+\varepsilon_L\\ &\le\Pi_L C_{\rm c}A_{\rm base} +\Pi_L\varepsilon_{\rm base}+\varepsilon_L.
    \end{aligned}
    \]
    The regularized capture score is no greater than its unregularized counterpart, giving the geometric factor. Conditional multiplication of success probabilities yields $p_0^L$; intersection with the consistency event and the union bound yield the final probability claim.
\end{proof}

\begin{remark}[Conditional coverage and scope]
    \label{rem:finite_coverage}
    The event $G_\ell$ concerns the \emph{final admissible block}, not an unscreened neuron. For example, set $Z_\ell=0$ on an unsuccessful full search and otherwise let it be the selected block's regularized capture score. Since $0\le Z_\ell\le \mathbf{1}_{G_\ell}$, a conditional bound $\mathbb E[Z_\ell\mid\mathcal F_\ell]\ge a>0$ implies \eqref{eq:finite_horizon_success} with $p_0=a$. Neither this coverage condition nor the two-point consistency bound follows automatically from residual-weighted sampling; in particular, consistency for the data-selected $u_L$ requires independent validation or suitable complexity control. The positive stopping threshold and possible search failure do not establish unconditional infinite-horizon convergence.
\end{remark}

\begin{remark}[Nonlinear updates]
    For solution-dependent Picard or Newton methods, the fixed-linear results describe the exact projection step for the corresponding frozen subproblem. After the operator is updated, the enlarged system is reassembled and the original nonlinear residual is recomputed for the acceptance test \eqref{eq:rbe_accept}.
\end{remark}

\begin{lemma}[One-step estimate for a frozen nonlinear operator]
    \label{lem:nonlinear_safeguard}
    Let $\mathcal F,\mathcal F_k:X\to Y$ act between normed spaces and satisfy $\mathcal F_k(u_k)=\mathcal F(u_k)$. Suppose the trial state $\widetilde u_{k+1}$ satisfies
    \[ \|\mathcal F_k(\widetilde u_{k+1})\|_Y \le\alpha_k\|\mathcal F(u_k)\|_Y,\quad 0\le\alpha_k<1, \]
    and $\|\mathcal F(\widetilde u_{k+1})-\mathcal F_k(\widetilde u_{k+1})\|_Y \le\omega_k\|\widetilde u_{k+1}-u_k\|_X$. Then
    \begin{equation}\label{eq:nonlinear_one_step_bound}
        \|\mathcal F(\widetilde u_{k+1})\|_Y \le\alpha_k\|\mathcal F(u_k)\|_Y +\omega_k\|\widetilde u_{k+1}-u_k\|_X.
    \end{equation}
    In particular, a non-increase of the true residual follows if
    \begin{equation}\label{eq:nonlinear_one_step_accept}
        \alpha_k\|\mathcal F(u_k)\|_Y +\omega_k\|\widetilde u_{k+1}-u_k\|_X \le\|\mathcal F(u_k)\|_Y.
    \end{equation}
    For the Newton model $\mathcal F_k(v)=\mathcal F(u_k) +\mathcal F'(u_k)(v-u_k)$, a Jacobian with Lipschitz constant $L_J$ on the trial segment yields $\|\mathcal F(\widetilde u_{k+1})-\mathcal F_k(\widetilde u_{k+1})\|_Y \le\tfrac12L_J\|\widetilde u_{k+1}-u_k\|_X^2$.
\end{lemma}
\begin{proof}
    The first estimate is the triangle inequality. For Newton, with $d=\widetilde u_{k+1}-u_k$, write the operator difference as $\int_0^1[\mathcal F'(u_k+td)-\mathcal F'(u_k)]d\,dt$ and use $\int_0^1tL_J\|d\|_X^2dt=\tfrac12L_J\|d\|_X^2$.
\end{proof}

\subsection{Scope of the theoretical statements}

\label{sec:theory_scope}

Theorems~\ref{thm:rbe_residual_decrease}--\ref{thm:finite_horizon_sfrbe} apply to a single fixed affine residual map, including its right-hand side and collocation weights. The final-state consistency assumption is conditional; uniform statistical bounds of this type, of order $m\sqrt{\log m/N}$ in the width $m$ and the number of collocation points $N$, are available for fixed-width trial spaces with bounded hidden parameters and i.i.d.\ uniform collocation points~\cite{dang2024adaptive}. For Picard/Newton iterations, the fixed-linear projection identity concerns only a frozen subproblem, not the original nonlinear PDE; accepted nonlinear trials are separately checked by Lemma~\ref{lem:nonlinear_safeguard}. 

\section{Numerical experiments}
\label{sec:numerical}

\subsection{Experimental setup}
\label{sec:experiment_setup}
We test the two adaptive modules on localized and multiscale elliptic problems, variable-coefficient Darcy flow, and nonlinear space--time KdV and Burgers equations. In most experiments, the value error is measured by the relative \(L^2\) error:
\begin{equation}\label{eq:numerical_rel_l2_error}
    {\rm Rel.}\ L^2 = \frac{ \|u_h-u_{\rm ref}\|_{L^2(\Omega)} }{ \|u_{\rm ref}\|_{L^2(\Omega)} } .
\end{equation}

For the reported gradient errors, we use the full-coordinate $H^1$ norm
\begin{equation}\label{eq:numerical_rel_h1_error}
    \|v\|_{H^1(\Omega)}^2=\|v\|_{L^2(\Omega)}^2+\|\nabla_{\boldsymbol z}v\|_{L^2(\Omega)}^2, \qquad {\rm Rel.}\ H^1=\frac{\|u_h-u_{\rm ref}\|_{H^1(\Omega)}}{\|u_{\rm ref}\|_{H^1(\Omega)}} .
\end{equation}
Here \(u_h\) denotes the computed solution and \(u_{\rm ref}\) denotes either the exact solution, when available, or a high-accuracy reference solution generated by a reliable traditional numerical method. The vector $\boldsymbol z$ contains all independent coordinates, so the gradient includes the temporal derivative in space--time problems. The integrals in \eqref{eq:numerical_rel_l2_error} and \eqref{eq:numerical_rel_h1_error} are evaluated by numerical quadrature on the prescribed evaluation grid.

Here ``AD--RaNN'' denotes the global-scale adaptive-distribution initialization used as the initial space and a comparison baseline; it is distinct from the parameter search in Appendix~\ref{app:rbe_parameter_refinement}.

All numerical experiments are implemented in MATLAB, where the least-squares systems, residual evaluations, SFR updates, and RBE procedures are assembled and solved. The schematic neural-network and algorithmic illustrations are generated in Python using Matplotlib. The two-peak Poisson and periodic KdV results are averaged over ten MATLAB random seeds $0,\ldots,9$, whereas the other reported experiments use seed $4$ unless otherwise stated.

All spatial and temporal derivatives of the random features are evaluated analytically. Unless otherwise stated, ``Features'' and ``Total features'' denote the stage-wise and cumulative numbers of raw randomized features, respectively.

\subsection{Two-peak Poisson benchmark: ablation and error analysis}
\label{sec:twopeak_poisson}

\paragraph{Problem setup, numerical results, and ablation}
The first test considers the Poisson problem
\begin{equation}\label{eq:twopeak_poisson}
    \begin{aligned}
        -\Delta u &= f, && (x,y)\in\Omega,\\ u &= g, && (x,y)\in\partial\Omega,
    \end{aligned}
    \qquad \Omega=(0,1)^2.
\end{equation}
The source term \(f\) and the Dirichlet boundary data \(g\) are generated from the manufactured exact solution
\[
\begin{aligned}
    u &= \exp\left( -\frac{(x-0.30)^2+(y-0.35)^2}{0.02^2} \right) +0.85\exp\left( -\frac{(x-0.72)^2+(y-0.68)^2}{0.05^2} \right).
\end{aligned}
\]

The manufactured solution contains two localized Gaussian peaks and is used to test whether the proposed structure-aware mechanism can detect multiple separated residual structures and redistribute the randomized feature anchors accordingly. After the SFR refresh, RBE is applied to further resolve the remaining localized and oscillatory residual components.

\begin{table}[!ht]
    \centering
    \caption{Numerical results for the two-peak Poisson equation~\eqref{eq:twopeak_poisson} (mean / best / worst over seeds \(0,\ldots,9\)). ``Features'' denotes the cumulative raw randomized-feature count reported at each stage; SFR keeps a fixed-width trial space. The values after $\pm$ denote one standard deviation over the ten seeds. Times are mean stage wall-clock times.}
    \label{tab:poisson_twopeak}
    \renewcommand{\arraystretch}{1.18}
    \setlength{\tabcolsep}{8pt}
    \resizebox{\linewidth}{!}{%
    \begin{tabular}{cccccc}
        \toprule Stage & Features (avg.) & Time (s, avg.) & Rel. \(L^2\) (avg.) & Rel. \(L^2\) (best) & Rel. \(L^2\) (worst)\\
        \midrule AD--RaNN & \(3500\) & \(8.477\) & \(2.068\mathrm{e}{-02}\pm1.950\mathrm{e}{-03}\) & \(1.829\mathrm{e}{-02}\) & \(2.327\mathrm{e}{-02}\)\\
        \midrule SFR & \(3500\) & \(24.555\) & \(4.866\mathrm{e}{-03}\pm1.171\mathrm{e}{-03}\) & \(3.609\mathrm{e}{-03}\) & \(6.974\mathrm{e}{-03}\)\\
        \midrule RBE & \(6450\) & \(93.250\) & \(1.164\mathrm{e}{-06}\pm2.114\mathrm{e}{-06}\) & \(8.588\mathrm{e}{-08}\) & \(7.096\mathrm{e}{-06}\)\\
        \bottomrule
    \end{tabular}}
\end{table}

\begin{table}[!ht]
    \centering
    \small
    \setlength{\tabcolsep}{4pt}
    \caption{Mean accepted RBE features by activation family over Poisson seeds \(0,\ldots,9\).}
    \label{tab:poisson_family_usage}
    \begin{tabular}{@{}cccccc@{}}
        \toprule Family & \texttt{tanh} & \texttt{ridge\_gauss} & \texttt{sin} & \texttt{radial\_gauss} & \texttt{mexican\_hat} \\
        \midrule Mean accepted features & 380 & 370 & 830 & 480 & 890 \\
        \bottomrule
    \end{tabular}
\end{table}

\begin{figure}[!ht]
    \centering
    \includegraphics[width=0.32\textwidth]{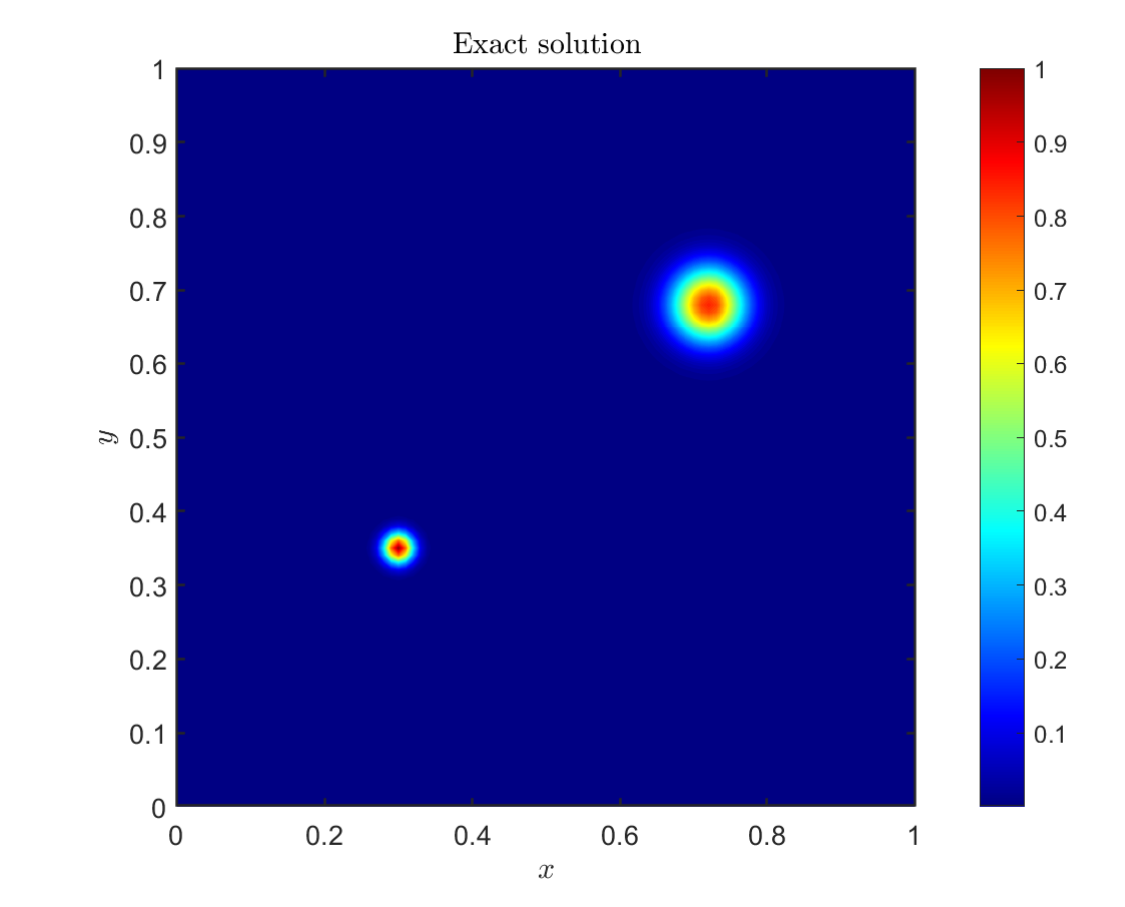}
    \includegraphics[width=0.32\textwidth]{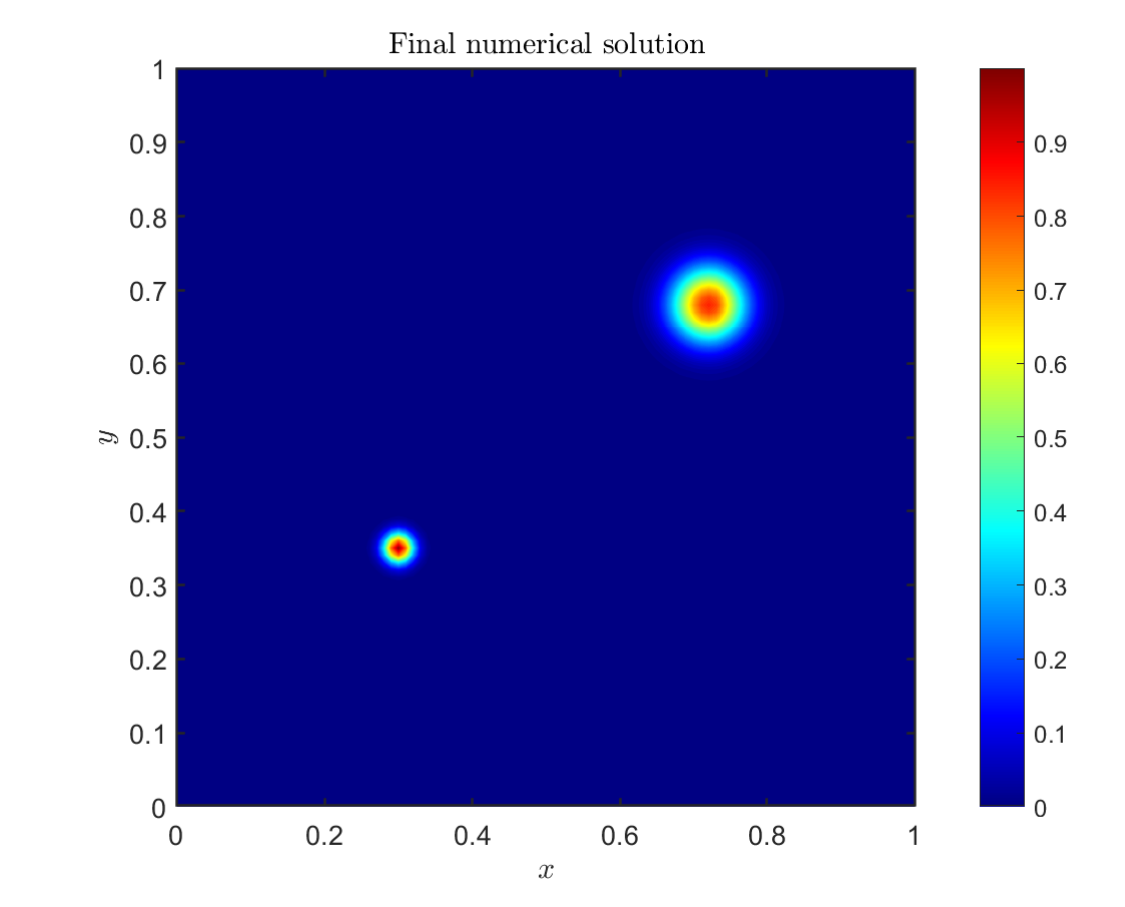}
    \includegraphics[width=0.32\textwidth]{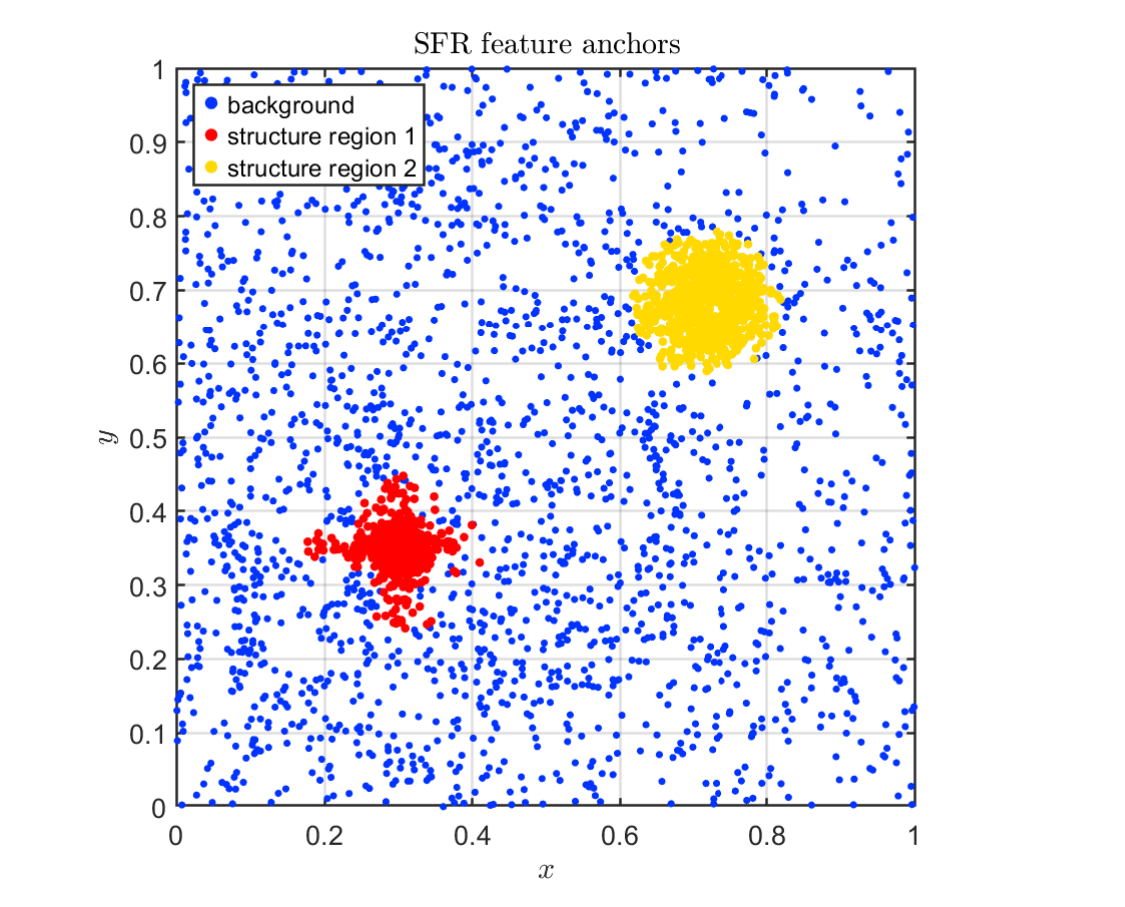}
    \caption{ Two-peak Poisson equation. Exact solution (left), final numerical prediction after RBE enrichment (middle), and SFR feature-anchor distribution (right). In the feature-anchor plot, different colors distinguish the detected structure regions and the background region. }
    \label{fig:poisson_twopeak_solution_center}
\end{figure}

\begin{figure}[!ht]
    \centering
    \includegraphics[width=0.32\textwidth]{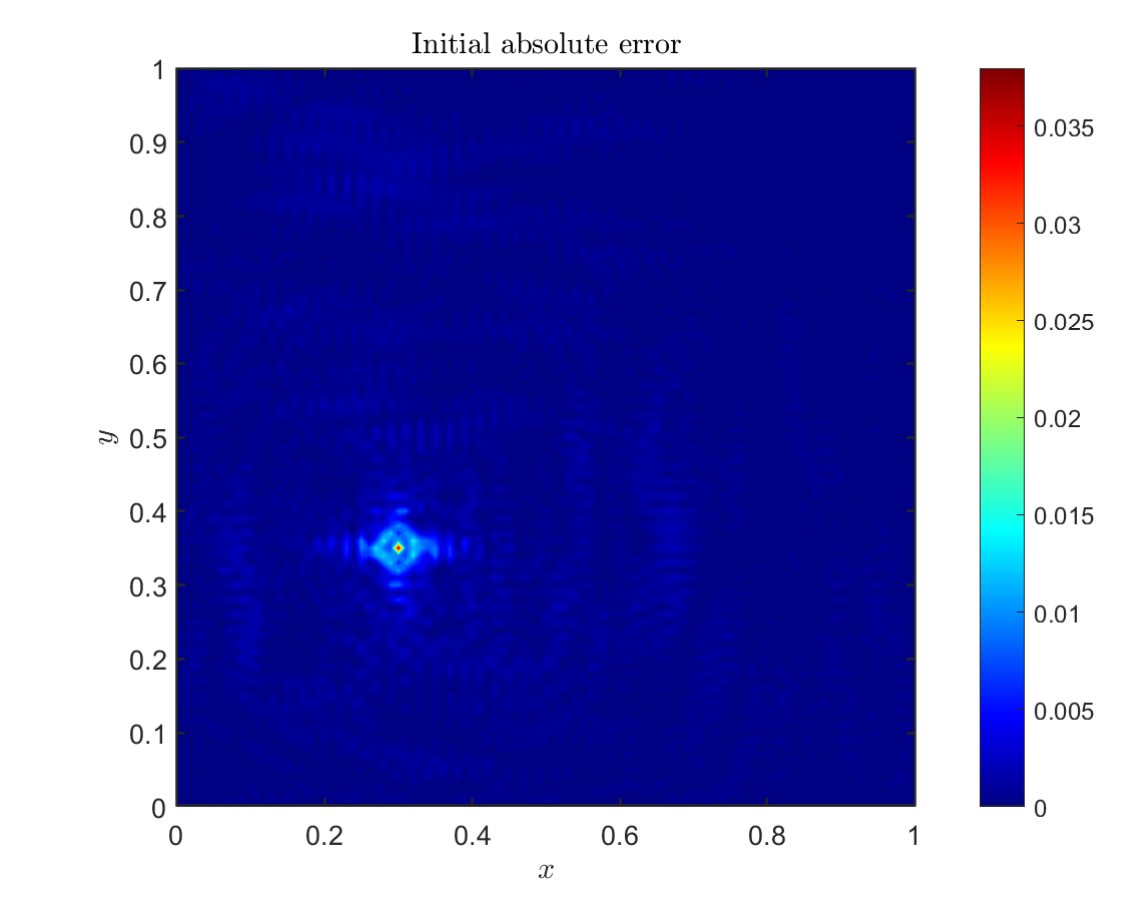}
    \includegraphics[width=0.32\textwidth]{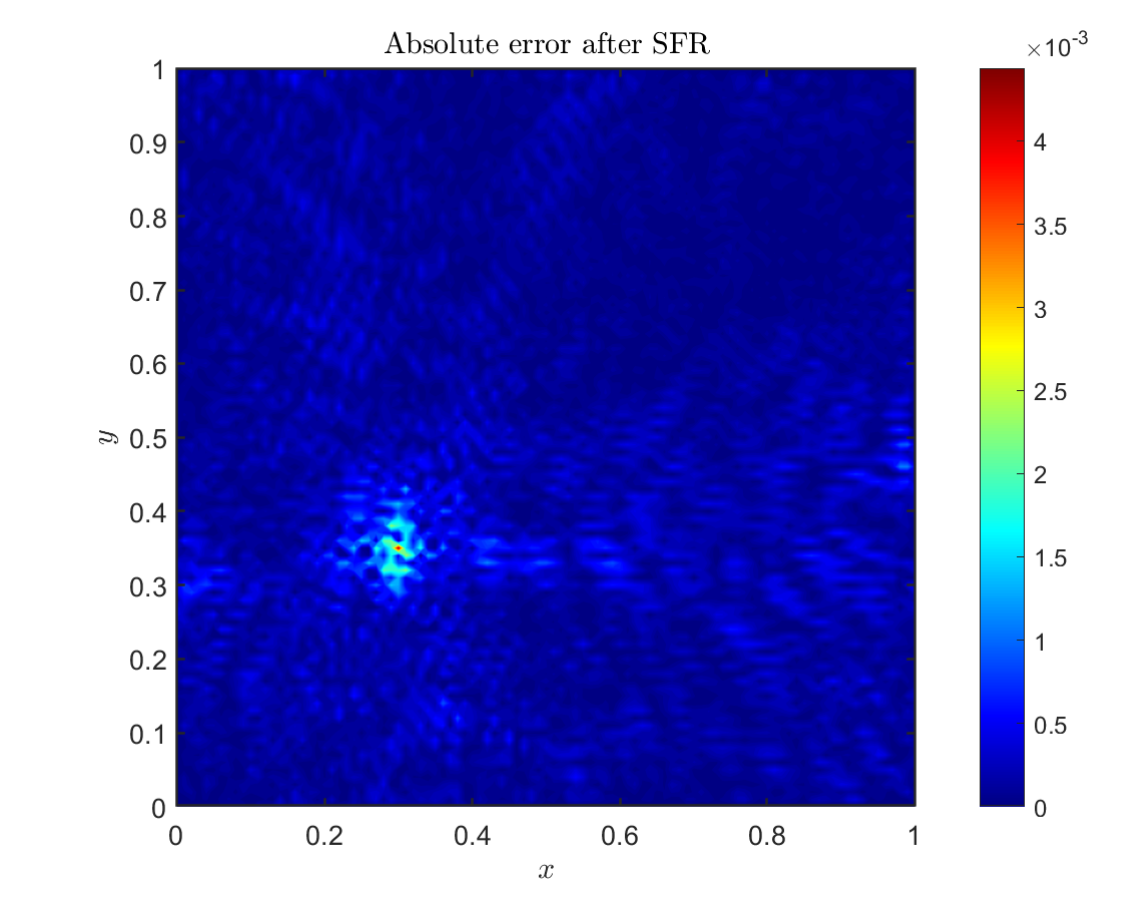}
    \includegraphics[width=0.32\textwidth]{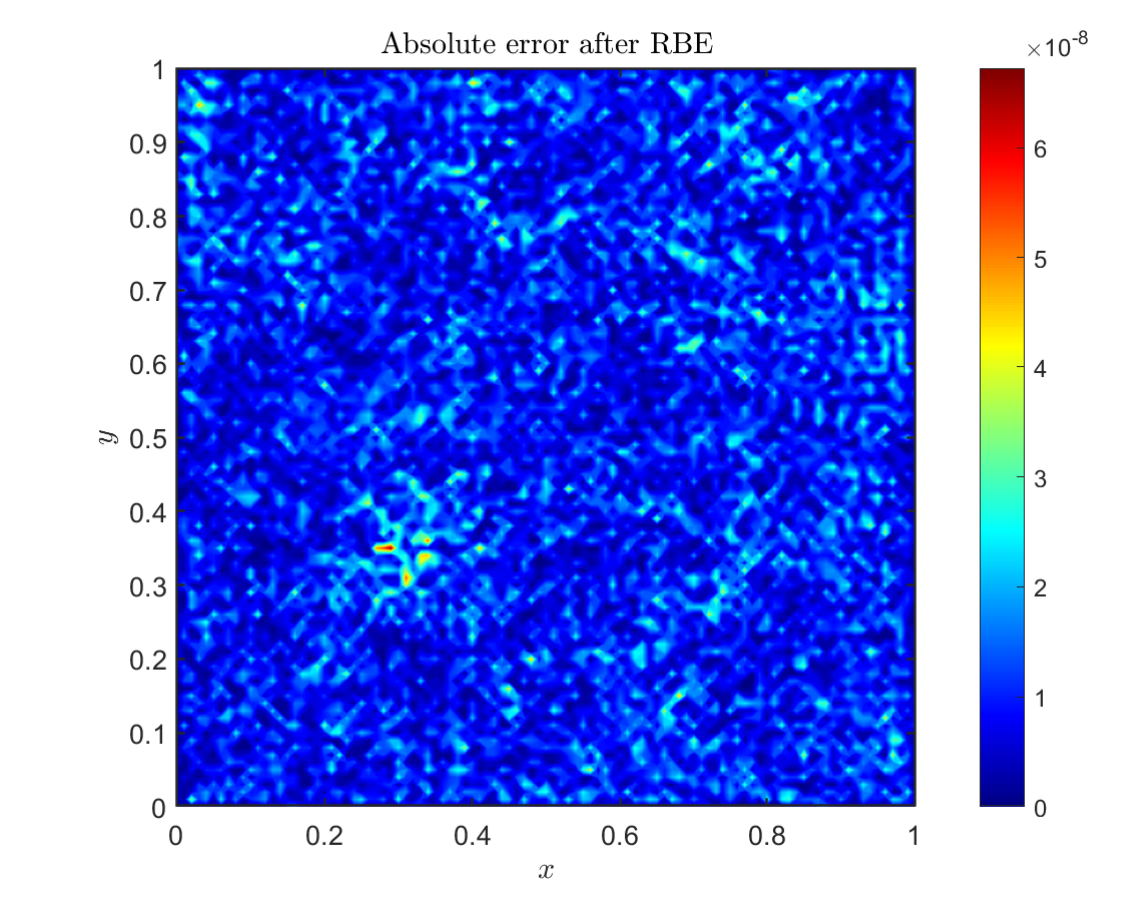}
    \caption{ Pointwise absolute error maps for the two-peak Poisson equation. From left to right: initial AD--RaNN, after SFR, and after RBE. }
    \label{fig:poisson_twopeak_error}
\end{figure}

As shown in Table~\ref{tab:poisson_twopeak}, the initial AD--RaNN state gives an average relative \(L^2\) error of \(2.068\times10^{-2}\) over MATLAB seeds \(0,\ldots,9\) (best \(1.829\times10^{-2}\), worst \(2.327\times10^{-2}\)). The SFR refresh reduces the average error to \(4.866\times10^{-3}\) (best \(3.609\times10^{-3}\), worst \(6.974\times10^{-3}\)) by reallocating the original randomized features between the detected structure regions and the background region. Starting from this SFR-refreshed approximation, RBE further decreases the average relative \(L^2\) error to \(1.164\times10^{-6}\) (best \(8.588\times10^{-8}\), worst \(7.096\times10^{-6}\)). Figures~\ref{fig:poisson_twopeak_solution_center} and~\ref{fig:poisson_twopeak_error} visualize the final prediction, the SFR feature-anchor distribution, and the progressive reduction of the pointwise error.

\begin{table}[!ht]
    \centering
    \caption{Comparison and mechanism ablation on the two-peak Poisson equation. The upper panel compares SFRBE--RaNN with representative external baselines. The lower panel isolates the effects of global initialization, fixed-budget feature redistribution, and residual-driven enrichment, with the wider SFR--RaNN row serving as a capacity control. Times and errors for the AD--RaNN, SFR--RaNN, RBE--RaNN (no SFR), and SFRBE--RaNN rows are means over seeds $0,\ldots,9$; the remaining rows are representative comparison or capacity-control results. The columns report absolute and relative $L^2$ errors, together with the corresponding errors measured in the full $H^1$ norm.}
    \label{tab:twopeak_main_comparison}
    \scriptsize
    \renewcommand{\arraystretch}{1.12}
    \setlength{\tabcolsep}{6pt}

    \resizebox{\linewidth}{!}{%
    \begin{tabular}{@{}lcccccc@{}}
        \toprule \multicolumn{7}{@{}l}{\textit{(a) Comparison with external baselines}}\\[-2pt]
        \midrule Method & DOFs / features & Time (s) & Abs. $L^2$ & Rel. $L^2$ & Abs. $H^1$ & Rel. $H^1$\\
        \midrule Expert-guided PINN & \(462593\) params. & \(5062.382\) & \(1.366\mathrm{e}{-02}\) & \(2.321\mathrm{e}{-01}\) & \(7.375\mathrm{e}{-02}\) & \(3.168\mathrm{e}{-02}\)\\
        HyResPINN & \(256711\) params. & \(6806.590\) & \(1.220\mathrm{e}{-03}\) & \(2.073\mathrm{e}{-02}\) & \(1.564\mathrm{e}{-02}\) & \(6.717\mathrm{e}{-03}\)\\
        Adaptive P1 FEM & \(7247310\) DOFs & \(279.380\) & \(2.029\mathrm{e}{-07}\) & \(3.447\mathrm{e}{-06}\) & \(2.670\mathrm{e}{-03}\) & \(1.147\mathrm{e}{-03}\)\\
        Adaptive P2 FEM & \(5724817\) DOFs & \(195.110\) & \(1.878\mathrm{e}{-08}\) & \(3.190\mathrm{e}{-07}\) & \(2.695\mathrm{e}{-05}\) & \(1.158\mathrm{e}{-05}\)\\
        AG--RaNN--1 & \(1200\) features & \(3.742\) & \(4.462\mathrm{e}{-04}\) & \(7.581\mathrm{e}{-03}\) & \(4.755\mathrm{e}{-02}\) & \(2.043\mathrm{e}{-02}\)\\
        AG--RaNN--2 & \(6450\) features & \(40.203\) & \(2.246\mathrm{e}{-04}\) & \(3.816\mathrm{e}{-03}\) & \(3.704\mathrm{e}{-02}\) & \(1.591\mathrm{e}{-02}\)\\
        Adaptive RFM & \(6450\) features & \(31.204\) & \(1.731\mathrm{e}{-04}\) & \(2.947\mathrm{e}{-03}\) & \(1.210\mathrm{e}{-02}\) & \(5.212\mathrm{e}{-03}\)\\
        KAPI--ELM & \(6450\) features & \(178.071\) & \(4.993\mathrm{e}{-06}\) & \(8.503\mathrm{e}{-05}\) & \(2.195\mathrm{e}{-04}\) & \(9.457\mathrm{e}{-05}\)\\
        \bottomrule
    \end{tabular}%
    }

    \vspace{5pt}

    \resizebox{\linewidth}{!}{%
    \begin{tabular}{@{}lcccccc@{}}
        \toprule \multicolumn{7}{@{}l}{\textit{(b) Mechanism ablation and capacity controls}}\\[-2pt]
        \midrule Method & DOFs / features & Time (s) & Abs. $L^2$ & Rel. $L^2$ & Abs. $H^1$ & Rel. $H^1$\\
        \midrule AD--RaNN (matched) & \(6450\) features & \(14.489\) & \(3.472\mathrm{e}{-04}\) & \(5.897\mathrm{e}{-03}\) & \(7.354\mathrm{e}{-02}\) & \(3.160\mathrm{e}{-02}\)\\
        SFR--RaNN (ours) & \(6450\) features & \(81.414\) & \(2.676\mathrm{e}{-05}\) & \(4.545\mathrm{e}{-04}\) & \(5.719\mathrm{e}{-03}\) & \(2.458\mathrm{e}{-03}\)\\
        SFR--RaNN (ours, wider) & \(12000\) features & \(434.678\) & \(3.530\mathrm{e}{-07}\) & \(5.997\mathrm{e}{-06}\) & \(1.057\mathrm{e}{-04}\) & \(4.541\mathrm{e}{-05}\)\\
        RBE--RaNN (no SFR) & \(6630\) features (avg.) & \(116.075\) & \(4.849\mathrm{e}{-07}\) & \(8.236\mathrm{e}{-06}\) & \(9.213\mathrm{e}{-05}\) & \(3.959\mathrm{e}{-05}\)\\
        SFRBE--RaNN (ours) & \(6450\) features (avg.) & \(126.282\) & \(6.854\mathrm{e}{-08}\) & \(1.164\mathrm{e}{-06}\) & \(1.613\mathrm{e}{-05}\) & \(6.933\mathrm{e}{-06}\)\\
        \bottomrule
    \end{tabular}%
    }

\end{table}

\begin{figure}[!ht]
    \centering
    \includegraphics[width=0.24\textwidth]{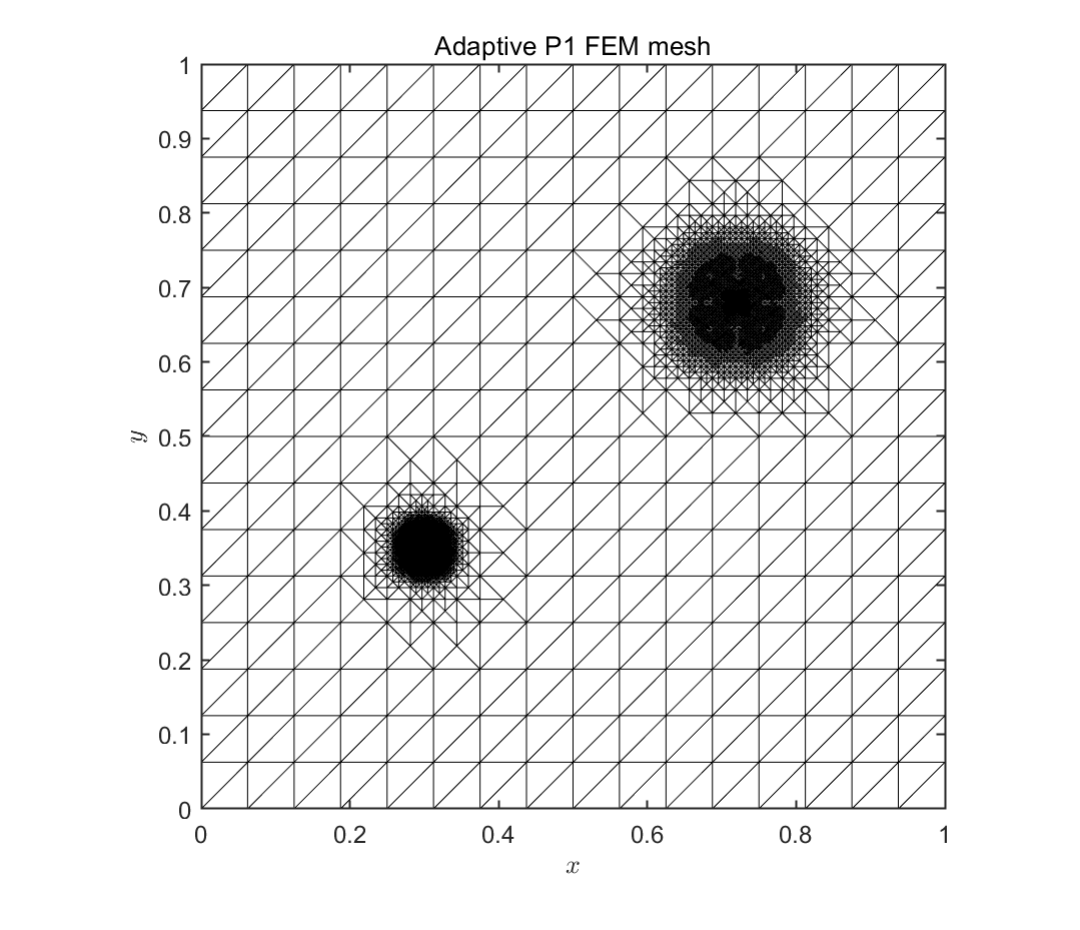}
    \includegraphics[width=0.24\textwidth]{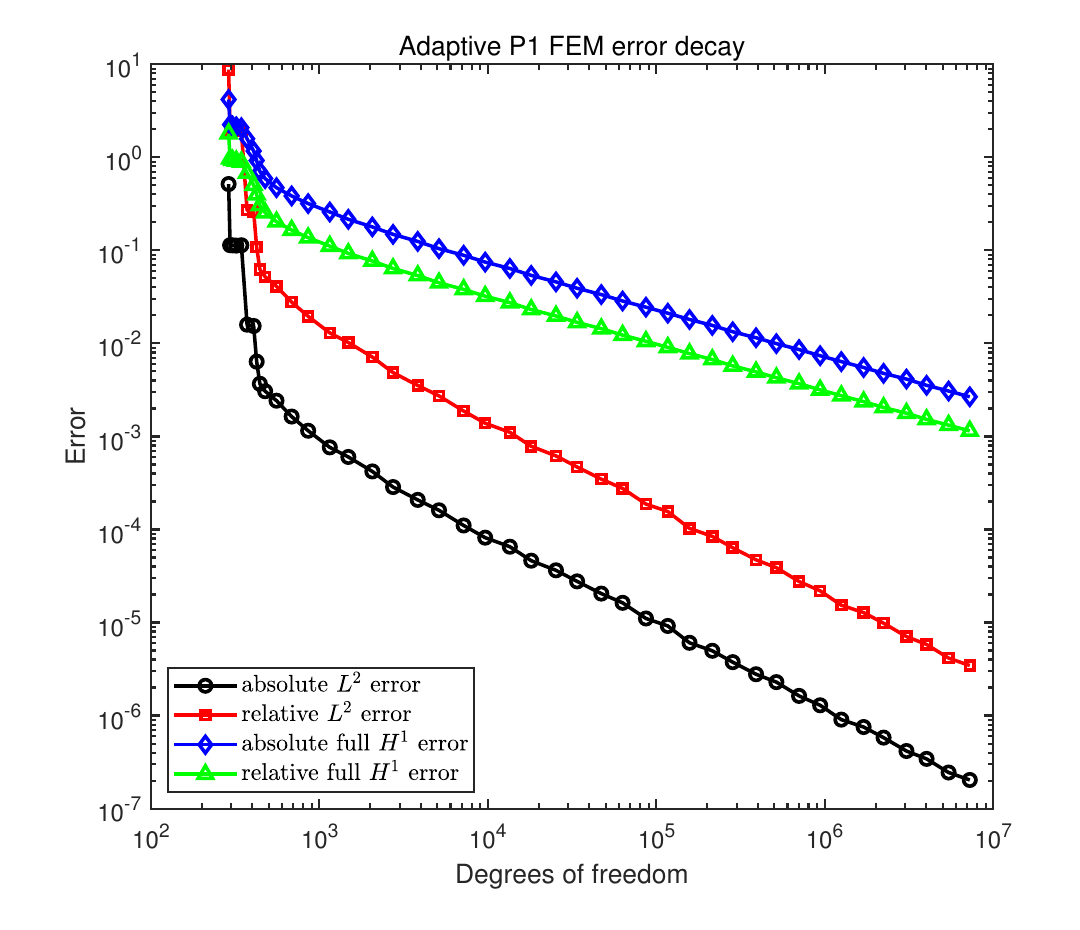}
    \includegraphics[width=0.24\textwidth]{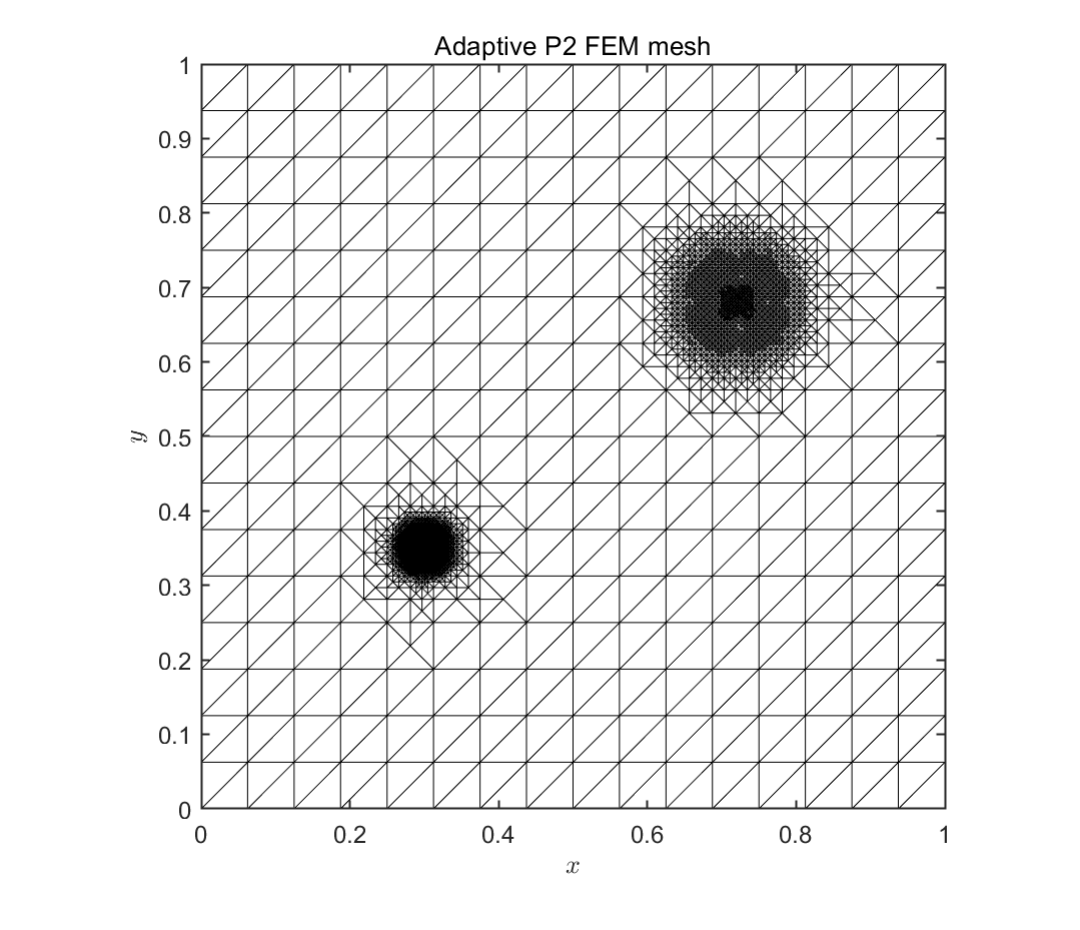}
    \includegraphics[width=0.24\textwidth]{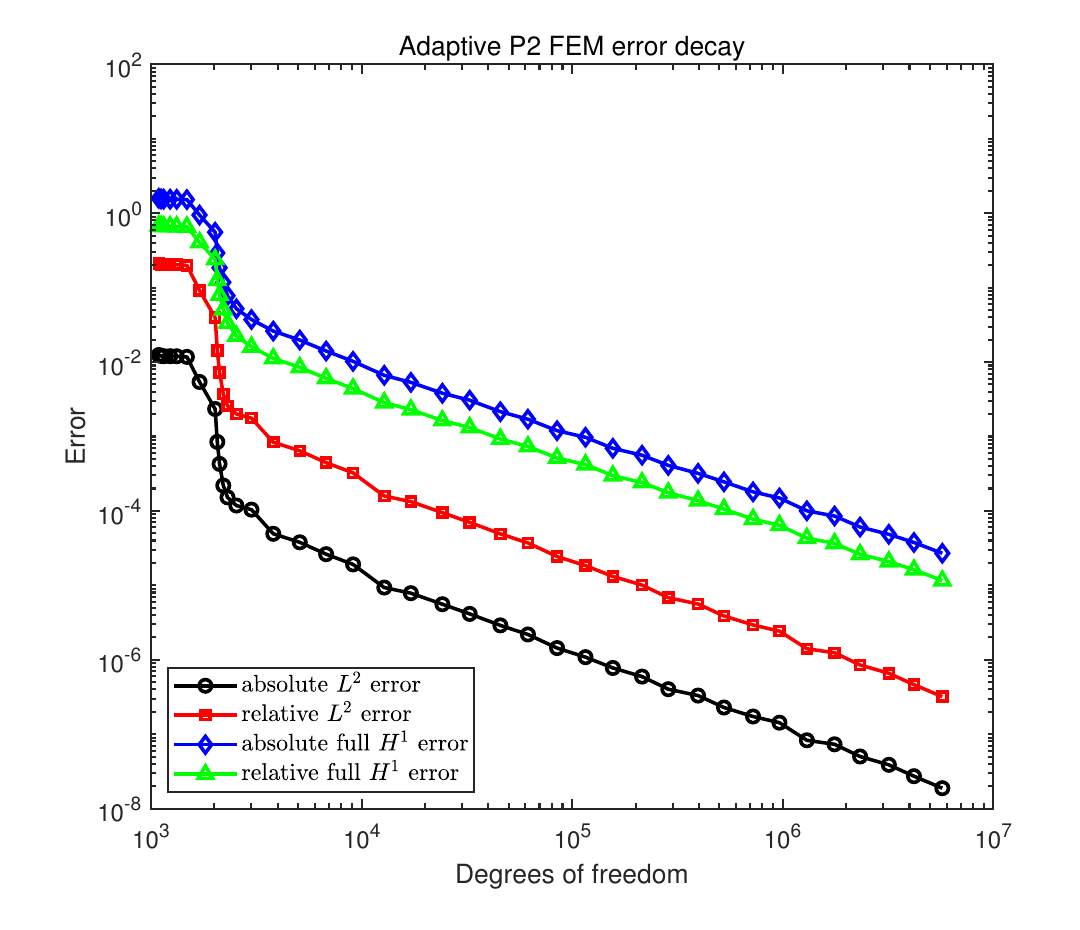}
    \caption{ Adaptive finite element results for the two-peak Poisson equation~\eqref{eq:twopeak_poisson}. From left to right: adaptive P1 mesh, absolute and relative \(L^2\) and full \(H^1\) error decay for adaptive P1 FEM, adaptive P2 mesh, and the corresponding error decay for adaptive P2 FEM. Both meshes are shown at the 20th AFEM iteration. }
    \label{fig:twopeak_afem}
\end{figure}

The reported wall-clock times are implementation-dependent and are included only as practical reference values. Table~\ref{tab:twopeak_main_comparison} compares SFRBE--RaNN with Expert-guided PINN~\cite{wang2023expert}, HyResPINN~\cite{cooley2025hyrespinns}, adaptive finite-element baselines, AG--RaNN~\cite{dang2024adaptive}, Adaptive RFM~\cite{chen2022bridging}, and KAPI--ELM~\cite{dwivedi2025kernel}, together with feature-count-matched AD--RaNN, SFR--RaNN and RBE--RaNN ablations.

The adaptive finite-element results use the iFEM MATLAB package with the standard solve--estimate--mark--refine loop and multigrid-preconditioned conjugate gradients~\cite{chen2008ifem,dorfler1996convergent,verfurth2013posteriori}. The meshes in Figure~\ref{fig:twopeak_afem} concentrate around the two Gaussian peaks. Adaptive P2 FEM obtains \(3.190\times10^{-7}\) in relative \(L^2\) and \(1.158\times10^{-5}\) in full \(H^1\), using \(5.725\times10^{6}\) degrees of freedom.

SFRBE--RaNN obtains an average relative \(L^2\) error of \(1.164\times10^{-6}\) (best \(8.588\times10^{-8}\), worst \(7.096\times10^{-6}\) over seeds \(0\)--\(9\)) and an average full \(H^1\) error of \(6.933\times10^{-6}\), using an average final total of \(6450\) randomized features and \(126.282\) seconds of compute time, excluding figure generation. Feature counts and finite-element degrees of freedom are not directly equivalent; this comparison highlights the mesh-free representation and direct evaluation of SFRBE--RaNN rather than a claim of uniformly lower error.

At a matched final feature count, AD--RaNN, SFR--RaNN, and SFRBE--RaNN achieve mean relative \(L^2\) errors of ($5.897\times10^{-3}$), ($4.545\times10^{-4}$), and ($1.164\times10^{-6}$), respectively. These results compare the three feature-space configurations at comparable final widths. As a complementary ablation, RBE--RaNN (no SFR) starts from the same 3500-feature AD--RaNN space, accepts nine RBE blocks, and reaches a ten-seed mean of 6630 features with relative \(L^2\) error \(8.236\times10^{-6}\). Thus, RBE alone improves the initial trial space, while applying SFR first yields a further gain; the wider SFR--RaNN result is reported separately as a capacity control.

\paragraph{Poisson residual-to-error estimate}
\label{para:poisson_residual_bound}
For the two-peak Poisson problem \eqref{eq:twopeak_poisson}, define \(R(v)=-\Delta v-f\) and \(b(v)=v|_{\partial\Omega}-g\), with \(\|\cdot\|_\Omega=\|\cdot\|_{L^2(\Omega)}\) and \(\|\cdot\|_{\partial\Omega}=\|\cdot\|_{L^2(\partial\Omega)}\). For \(u^*,v\in H^2(\Omega)\), the following explicit stability estimate holds:
\begin{equation}\label{eq:poisson_l2_bound}
    \|v-u^*\|_{L^2(\Omega)} \le \frac{\|R(v)\|_{L^2(\Omega)}}{2\pi^2} + \frac{\|b(v)\|_{L^2(\partial\Omega)}}{\sqrt{2}}.
\end{equation}
The proof is given in Appendix~\ref{app:poisson_residual_bound}.

For the seed-4 two-peak Poisson run, let \(\rho_L=\|R(u_L)\|_{L^2(\Omega)}\), \(\beta_L=\|b(u_L)\|_{L^2(\partial\Omega)}\), and \(e_L=\|u_L-u^*\|_{L^2(\Omega)}\) denote the interior residual norm, boundary residual norm, and absolute \(L^2(\Omega)\) solution error, respectively. Define \(B_L=\rho_L/(2\pi^2)+\beta_L/\sqrt{2}\). For the accepted block from stage \(L\) to stage \(L+1\), let \(\varrho_L=\|\bm r_L\|_2\) be the weighted discrete residual norm. Then \(q_L^{\mathrm{obs}}=\varrho_{L+1}/\varrho_L\) and \(q_L^{\mathrm{pred}}=\sqrt{1-s_L^{(0)}}\). The observed and predicted ratios are \((0.332022,0.332024)\) for RBE 2 and \((0.638280,0.638279)\) for RBE 8. The explicit Poisson bound is conservative, with \(B_L/e_L\) remaining on the order of \(10^3\). Complete stagewise diagnostics are reported in Table~\ref{tab:app_poisson_diagnostics}.

\begin{table}[!ht]
    \centering
    \scriptsize
    \setlength{\tabcolsep}{2pt}
    \renewcommand{\arraystretch}{1.05}
    \caption{Residual-to-error and RBE projection diagnostics for the detailed seed-4 Poisson run.}
    \label{tab:app_poisson_diagnostics}
    \resizebox{\linewidth}{!}{%
    \begin{tabular}{@{}lccccccccc@{}}
        \toprule & \multicolumn{5}{c}{Poisson residual-to-error diagnostics} & \multicolumn{4}{c}{RBE projection diagnostics} \\
        \cmidrule(lr){2-6}\cmidrule(lr){7-10} Stage & $\rho_L$ & $\beta_L$ & $e_L$ & $B_L$ & $B_L/e_L$ & $\varrho_L$ & $s_L^{(0)}$ & $q_L^{\rm obs}$ & $q_L^{\rm pred}$ \\
        \midrule Initial & $3.7801\mathrm{e}{+01}$ & $2.0260\mathrm{e}{-04}$ & $1.0768\mathrm{e}{-03}$ & $1.9152\mathrm{e}{+00}$ & $1.78\mathrm{e}{+03}$ & -- & -- & -- & -- \\
        SFR & $1.0008\mathrm{e}{+01}$ & $1.1999\mathrm{e}{-04}$ & $2.4459\mathrm{e}{-04}$ & $5.0710\mathrm{e}{-01}$ & $2.07\mathrm{e}{+03}$ & $5.5174\mathrm{e}{-02}$ & -- & -- & -- \\
        RBE 1 & $1.1027\mathrm{e}{-01}$ & $9.0593\mathrm{e}{-07}$ & $2.2659\mathrm{e}{-06}$ & $5.5870\mathrm{e}{-03}$ & $2.47\mathrm{e}{+03}$ & $6.0005\mathrm{e}{-04}$ & 0.99988 & 0.010876 & 0.010954 \\
        RBE 2 & $3.6376\mathrm{e}{-02}$ & $3.7292\mathrm{e}{-07}$ & $1.2297\mathrm{e}{-06}$ & $1.8431\mathrm{e}{-03}$ & $1.50\mathrm{e}{+03}$ & $1.9923\mathrm{e}{-04}$ & 0.88976 & 0.332022 & 0.332024 \\
        RBE 3 & $1.7380\mathrm{e}{-02}$ & $1.4947\mathrm{e}{-07}$ & $4.1892\mathrm{e}{-07}$ & $8.8059\mathrm{e}{-04}$ & $2.10\mathrm{e}{+03}$ & $9.4681\mathrm{e}{-05}$ & 0.77416 & 0.475235 & 0.475226 \\
        RBE 4 & $9.7973\mathrm{e}{-03}$ & $9.4211\mathrm{e}{-08}$ & $1.9670\mathrm{e}{-07}$ & $4.9640\mathrm{e}{-04}$ & $2.52\mathrm{e}{+03}$ & $5.3544\mathrm{e}{-05}$ & 0.68019 & 0.565520 & 0.565517 \\
        RBE 5 & $5.7025\mathrm{e}{-03}$ & $3.9802\mathrm{e}{-08}$ & $9.1348\mathrm{e}{-08}$ & $2.8892\mathrm{e}{-04}$ & $3.16\mathrm{e}{+03}$ & $3.0929\mathrm{e}{-05}$ & 0.66633 & 0.577637 & 0.577642 \\
        RBE 6 & $3.3086\mathrm{e}{-03}$ & $2.4162\mathrm{e}{-08}$ & $5.4388\mathrm{e}{-08}$ & $1.6763\mathrm{e}{-04}$ & $3.08\mathrm{e}{+03}$ & $1.7960\mathrm{e}{-05}$ & 0.66282 & 0.580685 & 0.580672 \\
        RBE 7 & $1.4387\mathrm{e}{-03}$ & $1.1447\mathrm{e}{-08}$ & $2.1626\mathrm{e}{-08}$ & $7.2893\mathrm{e}{-05}$ & $3.37\mathrm{e}{+03}$ & $7.8229\mathrm{e}{-06}$ & 0.81027 & 0.435573 & 0.435580 \\
        RBE 8 & $9.2044\mathrm{e}{-04}$ & $6.4937\mathrm{e}{-09}$ & $1.2849\mathrm{e}{-08}$ & $4.6635\mathrm{e}{-05}$ & $3.63\mathrm{e}{+03}$ & $4.9932\mathrm{e}{-06}$ & 0.59260 & 0.638280 & 0.638279 \\
        \bottomrule
    \end{tabular}%
    }
\end{table}

\subsection{Multiscale and variable-coefficient elliptic problems}

\paragraph{Case 1: Multiscale Helmholtz equation}

We consider the Helmholtz equation
\begin{equation}\label{eq:helmholtz}
    -u_{xx}-u_{yy}-u=f, \qquad (x,y)\in(0,1)^2,
\end{equation}
with Dirichlet boundary data prescribed from the exact solution
\begin{equation}\label{eq:helmholtz_exact}
    \begin{aligned}
        u(x,y) &= \sin(3\pi x)\sin(2\pi y) + 0.5\sin(25\pi x)\sin(22\pi y)\\ &\quad + 0.2\sin(50\pi x)\sin(57\pi y).
    \end{aligned}
\end{equation}
This example contains several separated spatial frequencies and therefore tests the ability of the proposed enrichment strategy to resolve strongly oscillatory multiscale residuals. Unlike the two-peak Poisson example, the dominant residual is not concentrated in a small number of localized regions. Consequently, the SFR gate is not activated, and the improvement is mainly achieved by the subsequent RBE block enrichment.

\begin{figure}[!ht]
    \centering
    \includegraphics[width=0.24\textwidth]{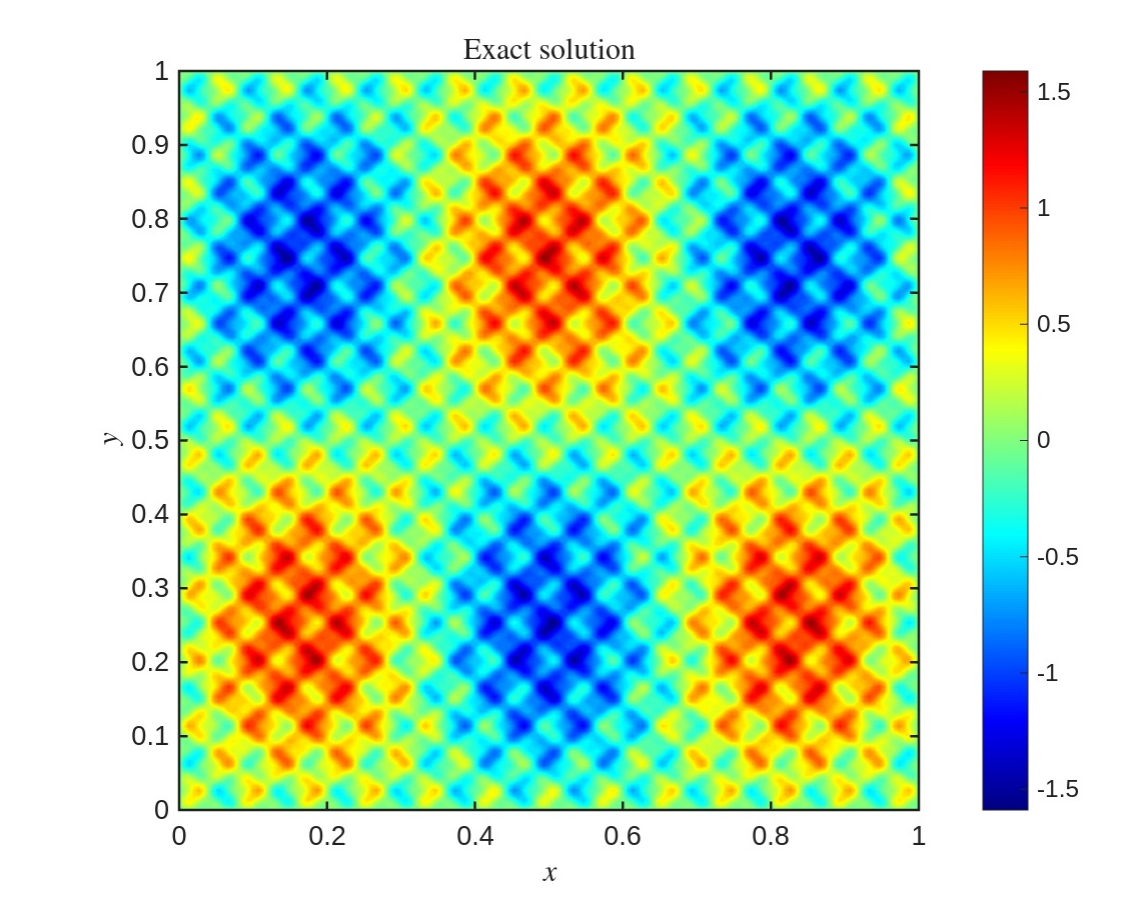}
    \hfill
    \includegraphics[width=0.24\textwidth]{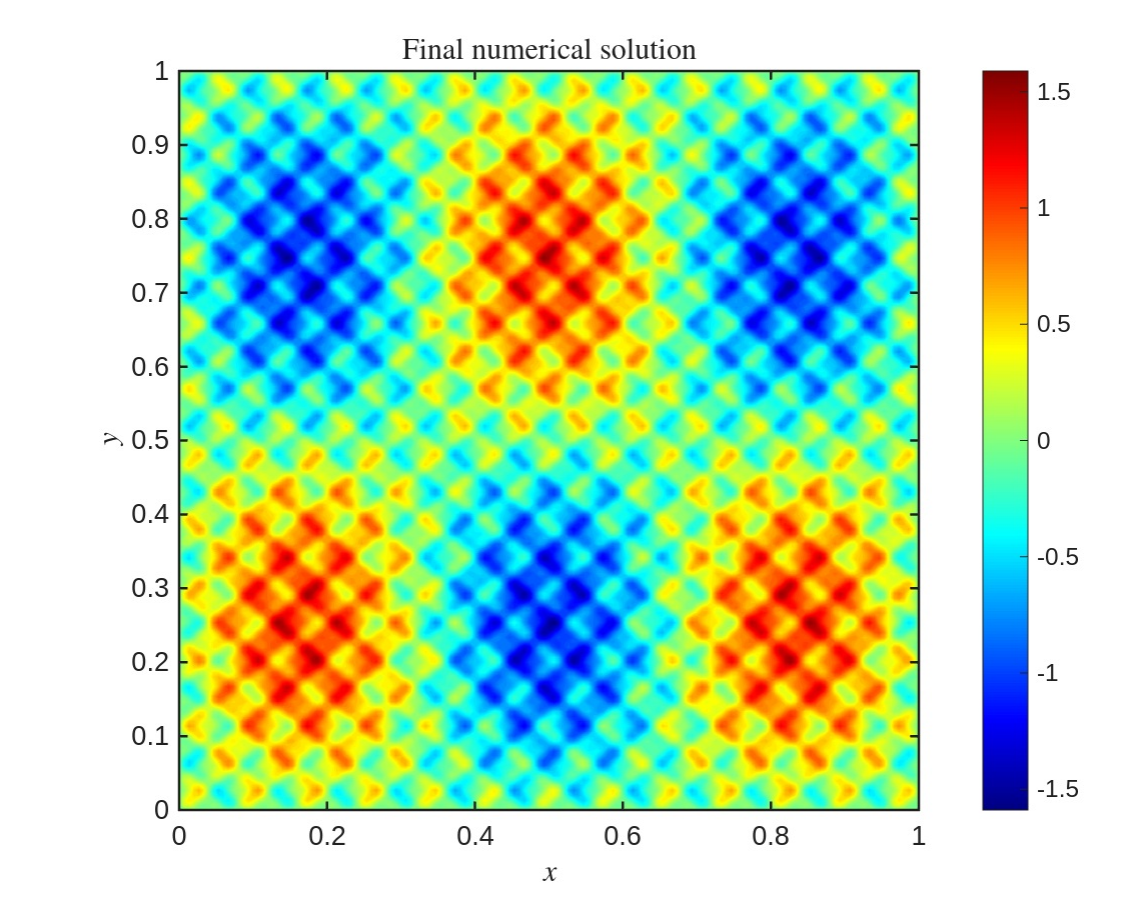}
    \hfill
    \includegraphics[width=0.24\textwidth]{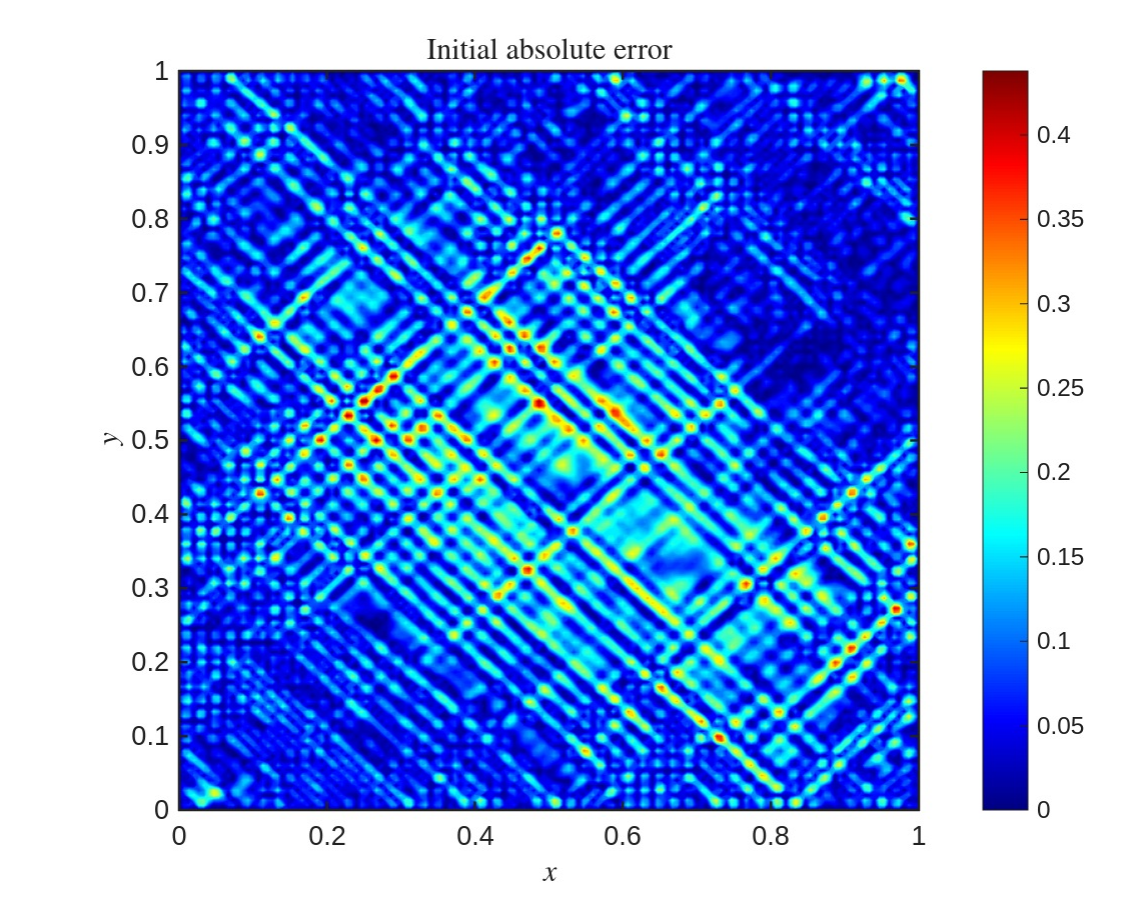}
    \hfill
    \includegraphics[width=0.24\textwidth]{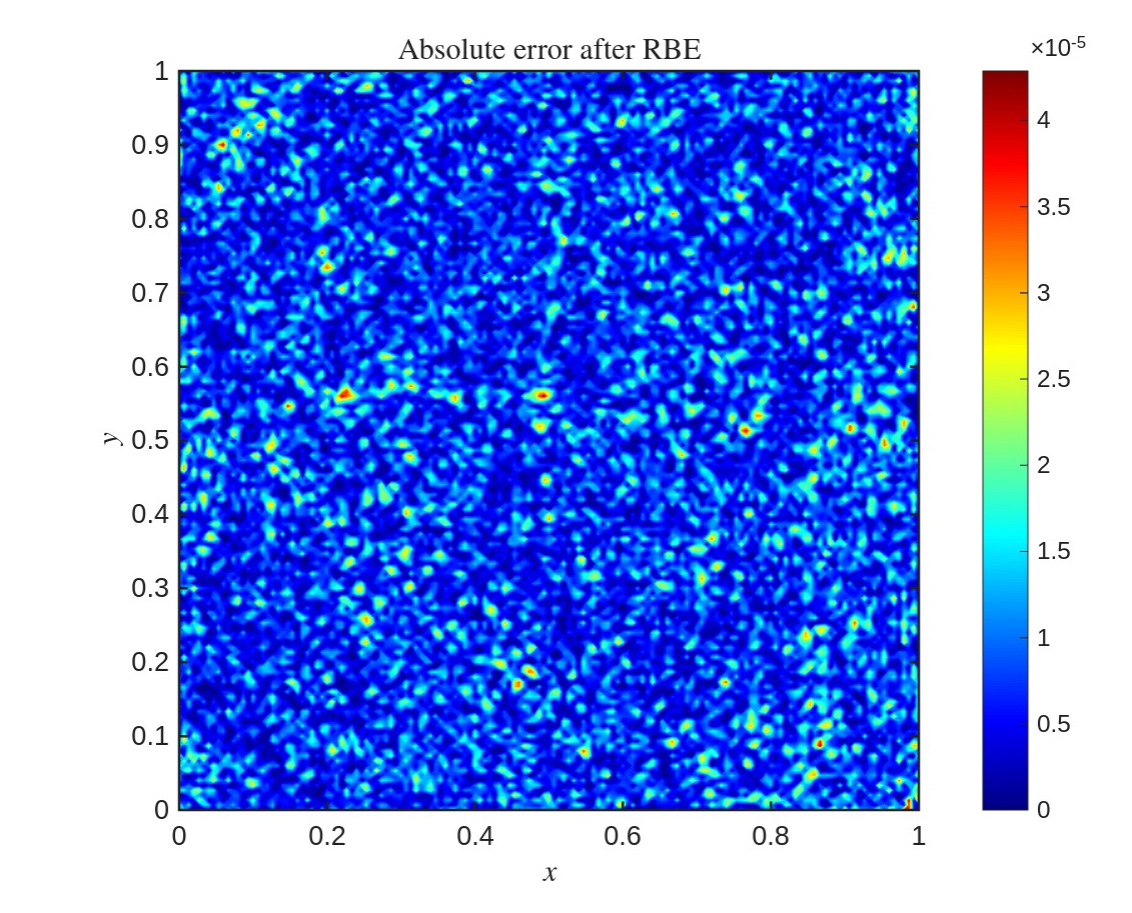}

    \caption{ Numerical results for the multiscale Helmholtz equation~\eqref{eq:helmholtz}. From left to right: exact solution, final numerical prediction after RBE enrichment, initial pointwise absolute error, and pointwise absolute error after RBE enrichment. }
    \label{fig:helmholtz_solution_error}
\end{figure}

The SFR gate remains inactive because the residual is globally oscillatory rather than localized. Starting from the 3500-feature AD--RaNN state \(\mathrm{Rel.}\,L^2=2.069\times10^{-1}\), RBE accepts ten blocks drawn from the families \texttt{sin}, \texttt{tanh}, \texttt{ridge\_gauss}, and \texttt{mexican\_hat}, reducing the relative \(L^2\) error to \(1.647\times10^{-5}\). The corresponding reduction in the pointwise error is shown in Figure~\ref{fig:helmholtz_solution_error}; the per-block scales are listed in Appendix~\ref{app:rbe_logs}.

\begin{table}[!ht]
    \centering
    \caption{Stage-wise results for the multiscale Helmholtz equation~\eqref{eq:helmholtz}. The SFR row records the inactive gate and unchanged 3500-feature allocation. ``Total features'' is cumulative, and the reported times are stage wall-clock values.}
    \label{tab:helmholtz}
    \small
    \setlength{\tabcolsep}{3.5pt}
    \renewcommand{\arraystretch}{1.08}
    \resizebox{\linewidth}{!}{%
    \begin{tabular}{ccccccc}
        \toprule Stage & Family & Features & Total features & Scale $\bm r$ & Time (s) & Rel. $L^2$\\
        \midrule Initial AD--RaNN & \texttt{tanh} & 3500 & 3500 & [63.667,66.557] & 15.303 & $2.069\mathrm{e}{-01}$\\
        \midrule SFR & Not triggered & -- & 3500 & -- & 10.247 & $2.069\mathrm{e}{-01}$\\
        \midrule RBE & \texttt{sin}/\texttt{tanh}/\texttt{ridge\_gauss}/\texttt{mexican\_hat} & 5600 (10 blocks) & 9100 & {--} & 501.654 & $1.647\mathrm{e}{-05}$\\
        \bottomrule
    \end{tabular}}
\end{table}

\paragraph{Case 2: Multiscale Darcy equation}
We next consider the linear variable-coefficient Darcy problem
\begin{equation}\label{eq:darcy_phi}
    \begin{gathered}
        \begin{aligned}
            -\nabla\!\cdot\!\left(K_\sigma(x,y)\nabla u\right)&=1, && (x,y)\in(0,1)^2,\\ u(0,y)=1,u(1,y)&=0, &&0\le y\le1,\\ K_\sigma\nabla u\cdot\boldsymbol n&=0, &&y\in\{0,1\},
        \end{aligned}
        \\[-2pt]
        \begin{aligned}
            K_\sigma(x,y)&=\exp\!\Bigl(5\sin(4\pi x)\sin(4\pi y) +2.5\sin(10\pi x)\sin(6\pi y)\\[-2pt] &\qquad +7\exp\!\left[-\frac{(x-0.7)^2+(y-0.3)^2}{\sigma^2}\right]\Bigr), \qquad \sigma=0.03.
        \end{aligned}
    \end{gathered}
\end{equation}
This is a linear problem with two oscillatory coefficient scales and a localized, smooth, high-contrast feature. On the evaluation grid, $K_\sigma$ ranges from approximately $9.5\times10^{-4}$ to $1.1\times10^{3}$, so the coefficient contrast is about $1.1\times10^{6}$. The resulting solution gradient and flux are strongly nonuniform even though the coefficient remains smooth; this makes the case a compact test of regional redistribution and residual-driven enrichment without introducing an interface condition. The reference field is computed by quadratic finite elements with $H_{\max}=10^{-3}$ and stored on a $101\times101$ evaluation grid.

\begin{table}[!ht]
    \centering
    \caption{Stage-wise results for the Multiscale Darcy problem \eqref{eq:darcy_phi}. The two SFR rows report the structure- and background-region neuron counts $m_S$ and $m_B$; ``Total'' is the cumulative feature count. Times are stage wall-clock times, and the detailed RBE blocks are listed in Appendix~\ref{app:rbe_logs}. The run uses seed $4$ and times exclude figure generation.}
    \label{tab:darcy_stage}
    \small
    \setlength{\tabcolsep}{3.2pt}
    \renewcommand{\arraystretch}{1.08}
    \resizebox{\linewidth}{!}{%
    \begin{tabular}{ccccccc}
        \toprule Stage & Family & Features & Total & Scale $\boldsymbol r$ & Time (s) & Rel. $L^2$\\
        \midrule Initial AD--RaNN & \texttt{tanh} & 3000 & 3000 & [18.530,\,18.510] & 63.924 & $1.239\mathrm{e}{-01}$\\
        \midrule \multirow{2}{*}{SFR} & \texttt{tanh} & $m_S=1482$ & \multirow{2}{*}{3000} & \begin{tabular}{c}$\mathcal S_1:[34.875,\,35.573]$\\[-2pt]$\mathcal S_2:[35.572,\,36.671]$\end{tabular} & \multirow{2}{*}{24.560} & \multirow{2}{*}{$8.188\mathrm{e}{-02}$}\\[-1pt]
        & \texttt{tanh} & $m_B=1518$ & & [29.131,\,22.165] & & \\
        \midrule RBE & \texttt{mexican\_hat}/\texttt{ridge\_gauss}/\texttt{tanh} & 4900 (7 blocks) & 7900 & {--} & 141.411 & $8.362\mathrm{e}{-03}$\\
        \bottomrule
    \end{tabular}}
\end{table}
The initial AD--RaNN approximation has relative $L^2$ error $1.239\times10^{-1}$. SFR accepts the regional refresh without increasing the 3000-feature budget and reduces this error to $8.188\times10^{-2}$; seven accepted RBE blocks then reduce it to $8.362\times10^{-3}$. Figure~\ref{fig:darcy_phi_summary} shows the six saved diagnostics: the reference solution, final prediction, and SFR-anchor distribution, followed by the initial error, post-SFR error, and RBE error.
\begin{figure}[!ht]
    \centering
    \begin{tabular}{ccc}
        \includegraphics[width=0.30\textwidth]{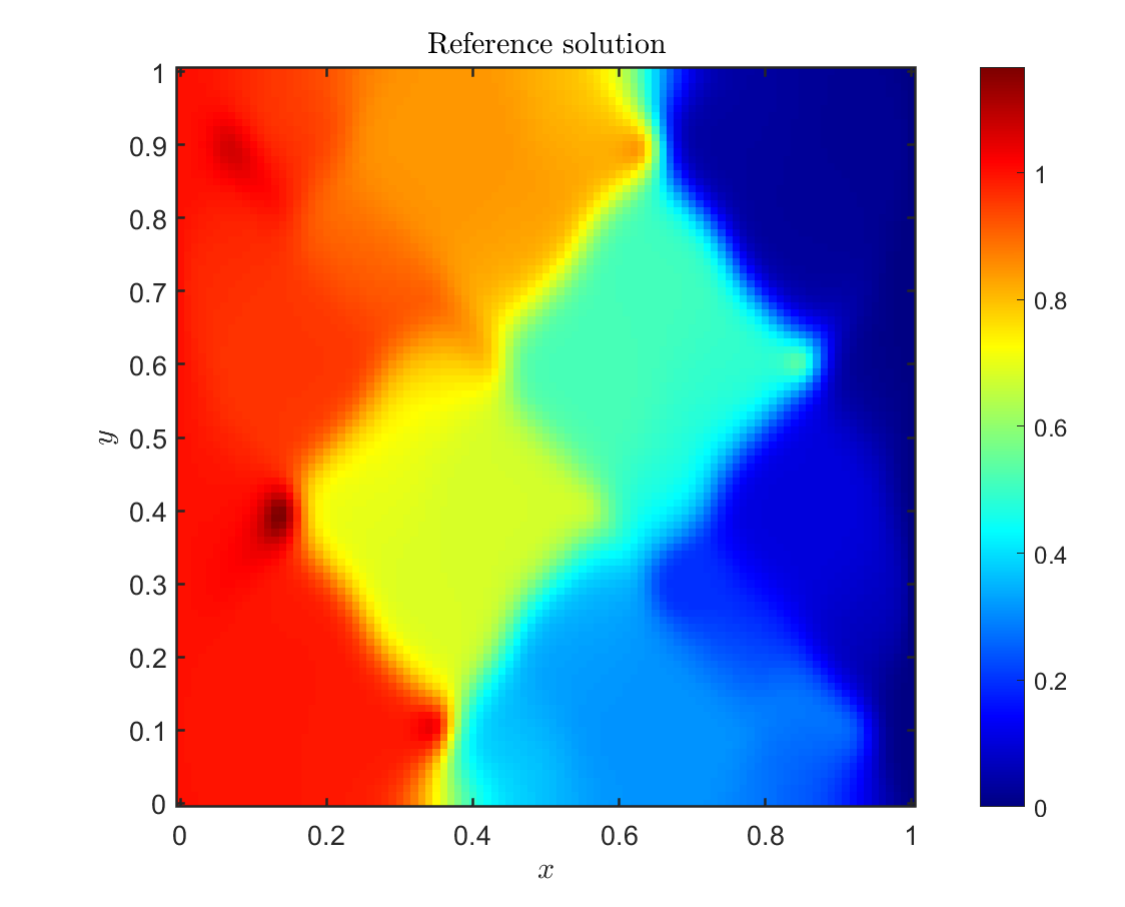} & \includegraphics[width=0.30\textwidth]{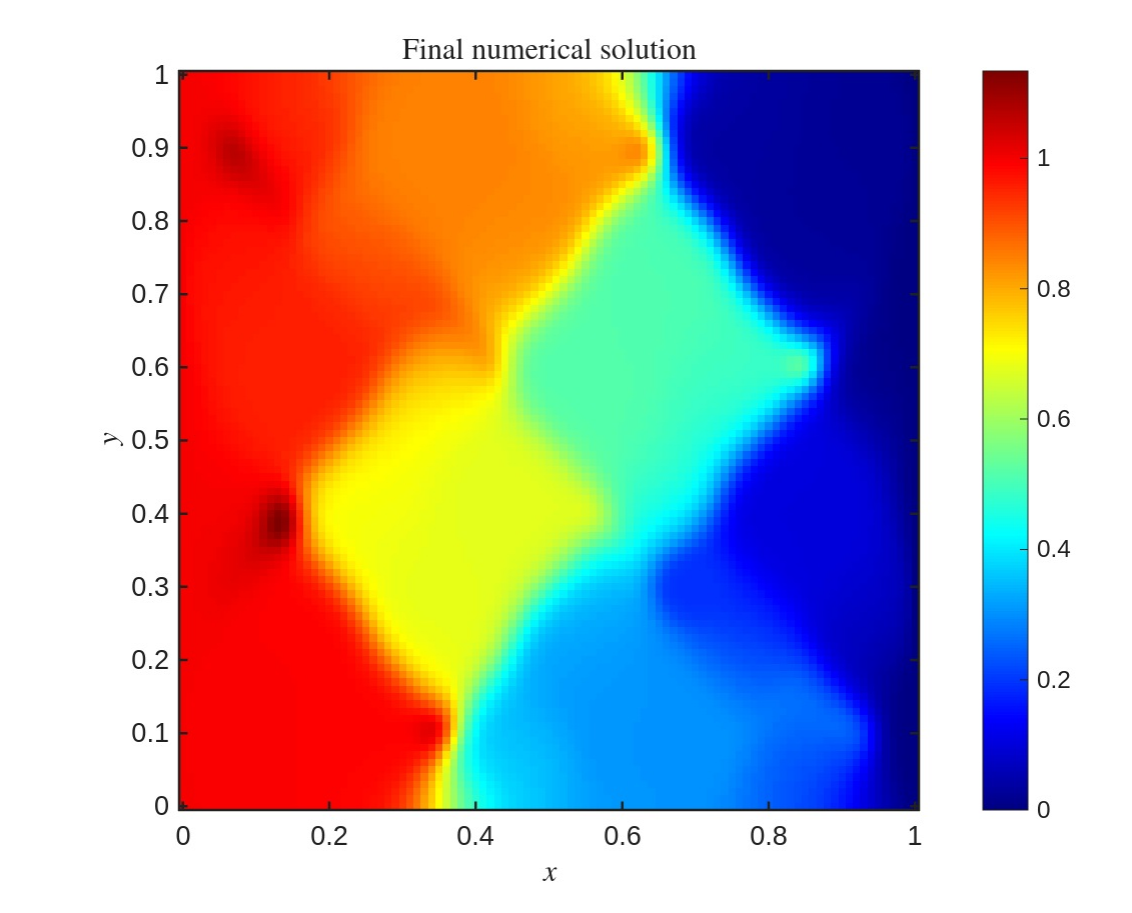} & \includegraphics[width=0.30\textwidth]{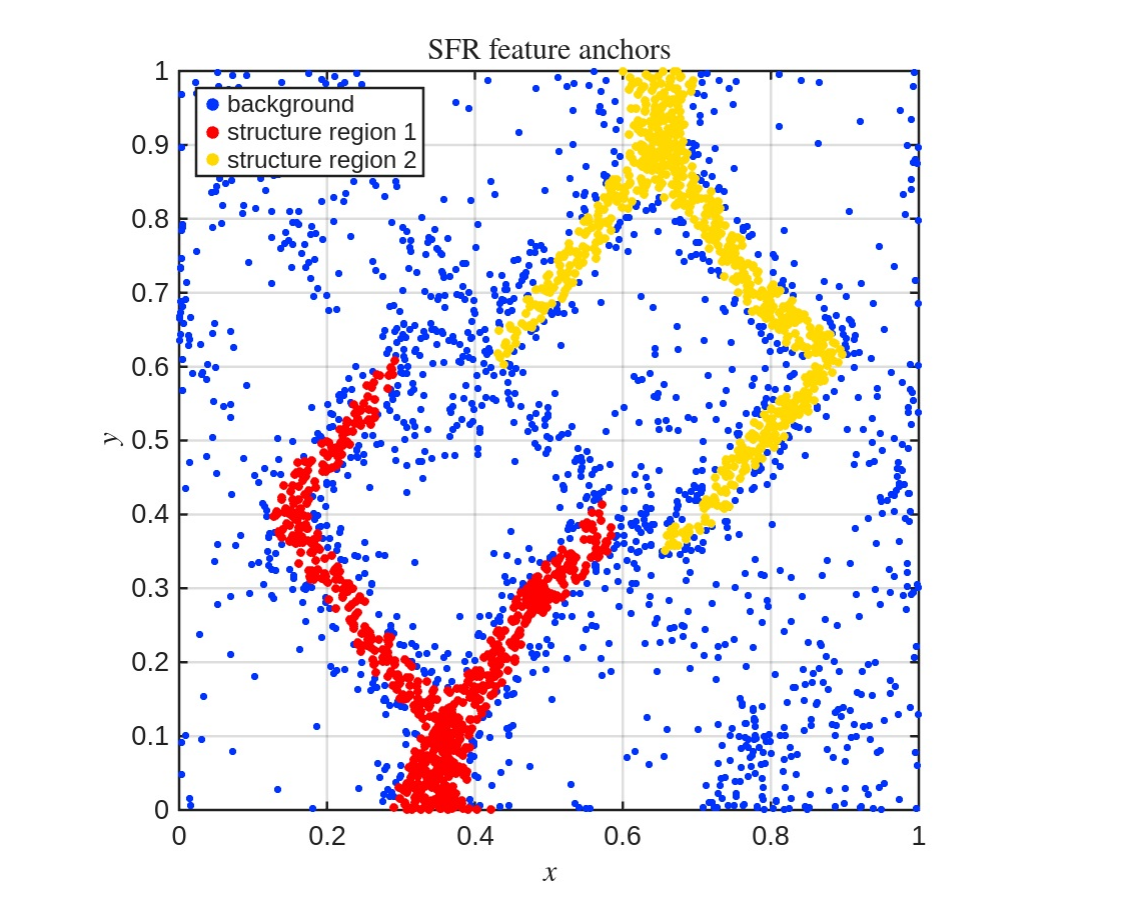}\\[-1pt]
        \includegraphics[width=0.30\textwidth]{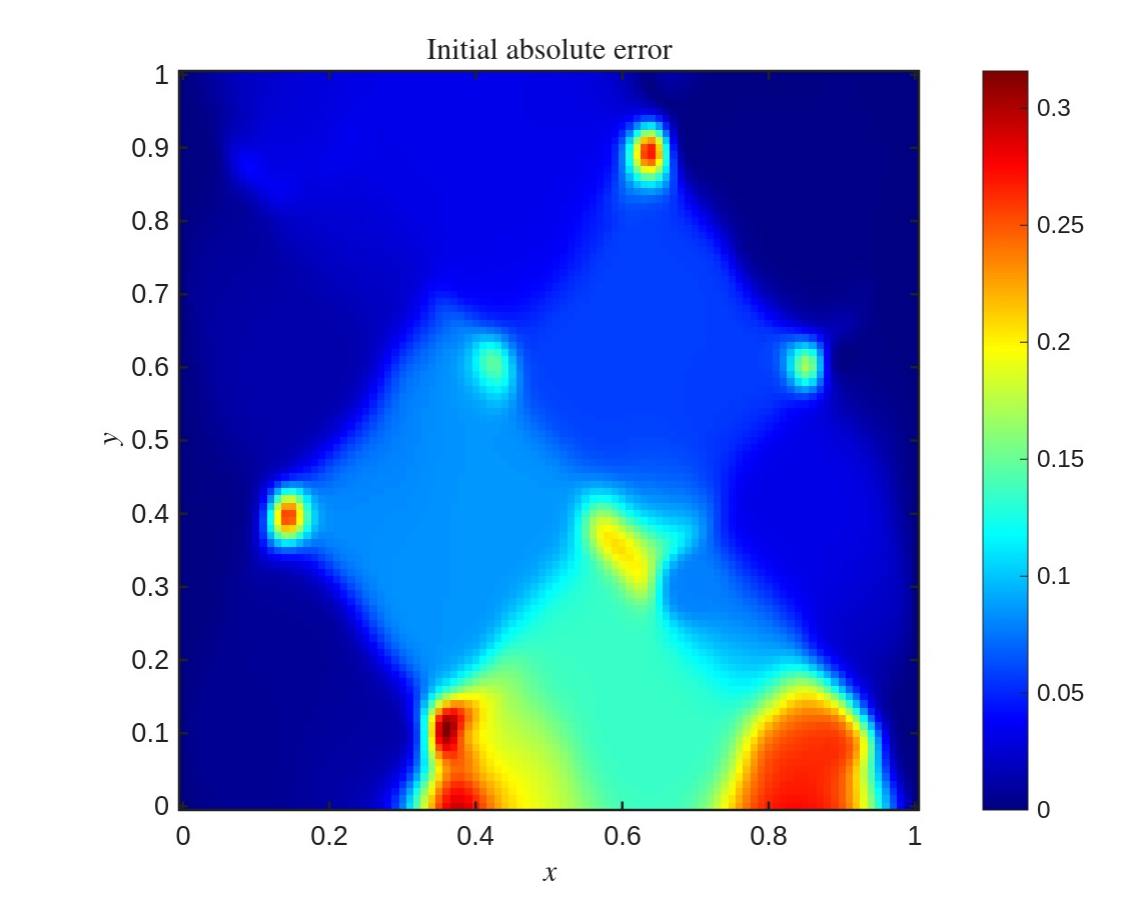} & \includegraphics[width=0.30\textwidth]{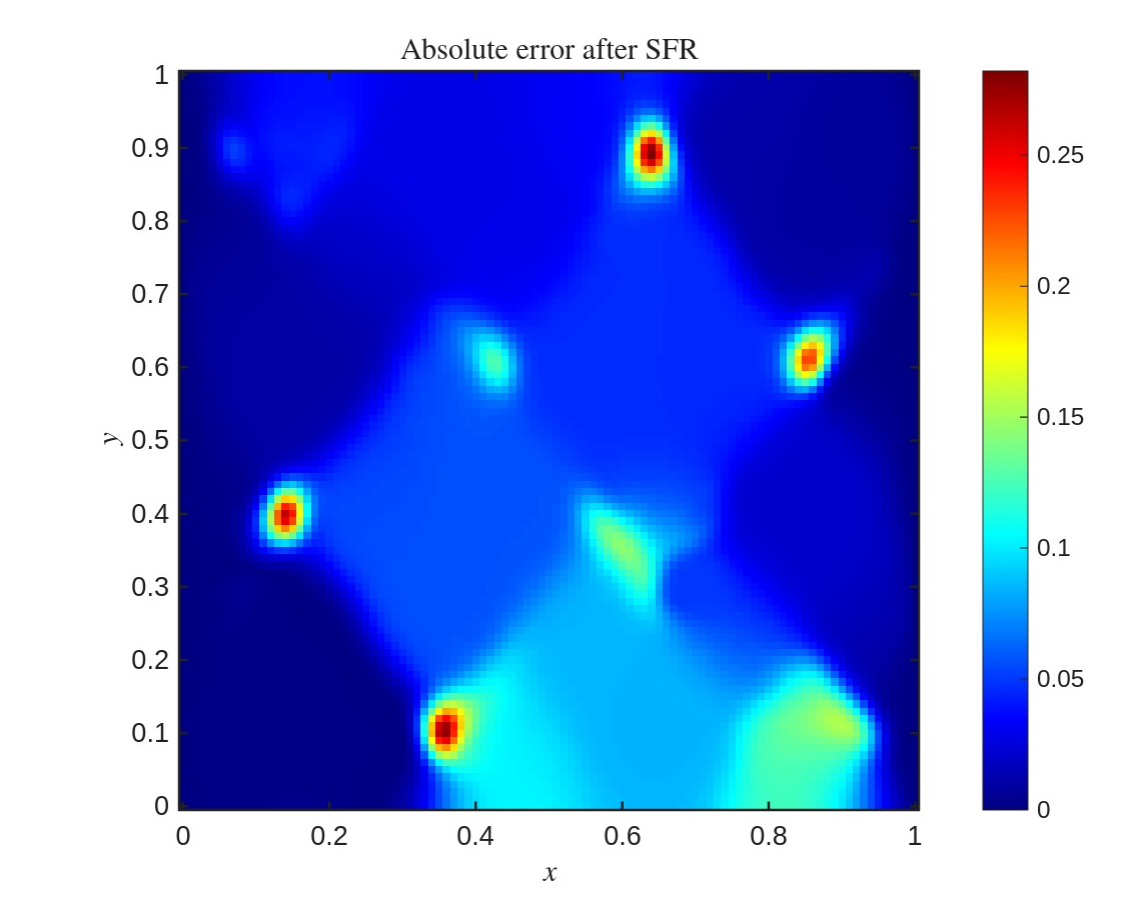}& \includegraphics[width=0.30\textwidth]{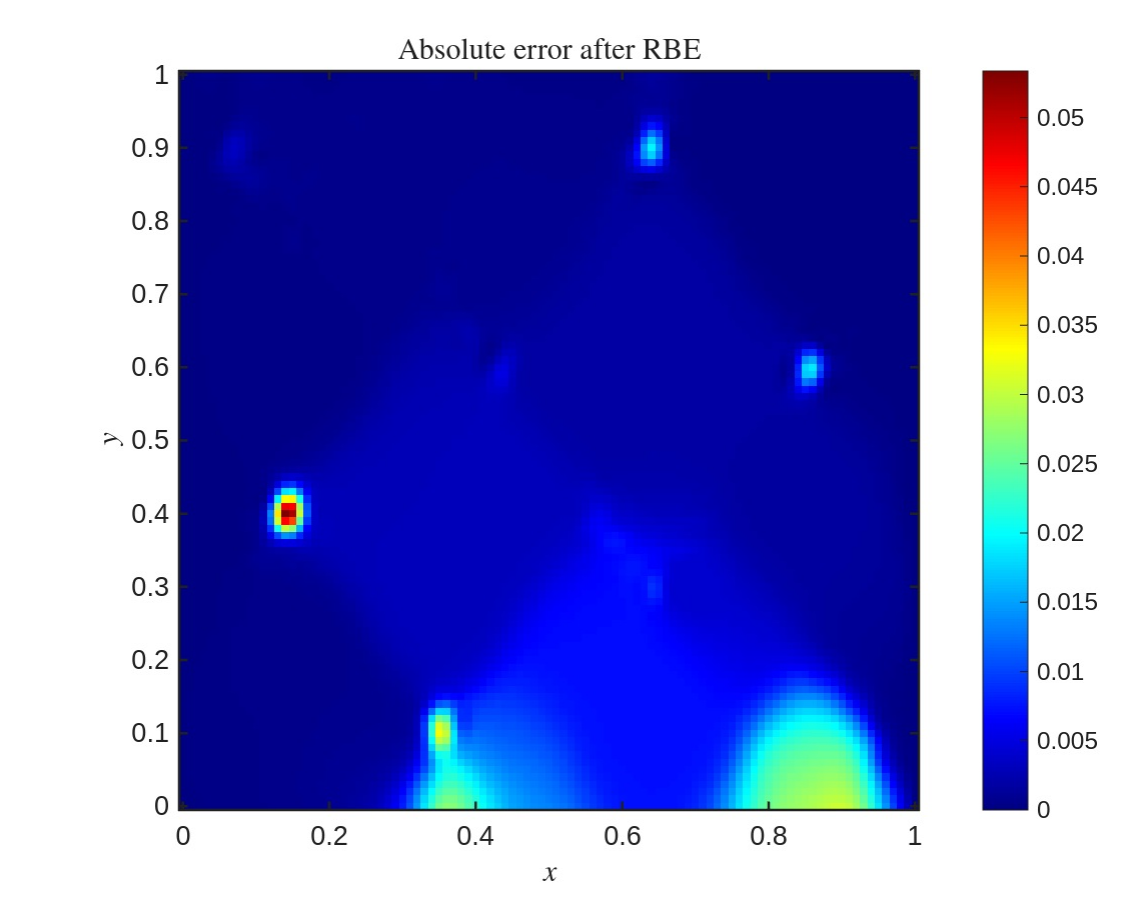}\\[-1pt]
    \end{tabular}
    \caption{Multiscale Darcy results for \eqref{eq:darcy_phi}. The first row shows the reference solution, the final numerical solution after RBE enrichment, and the SFR feature-anchor distribution. The second row shows the initial pointwise absolute error, the pointwise absolute error after SFR, and the pointwise absolute error after RBE.}
    \label{fig:darcy_phi_summary}
\end{figure}

\subsection{Nonlinear space--time problems}

\paragraph{Case 1: Generalized KdV localized traveling wave}

We consider the generalized KdV equation in the space--time domain:
\begin{equation}\label{eq:kdv}
    u_t+u_x+u^3u_x+\epsilon u_{xxx}=0, \qquad (x,t)\in(-2,3)\times(0,2),
\end{equation}
where $\epsilon=0.2058\times10^{-4}$. The exact solution and its parameters are written compactly as
\begin{equation}\label{eq:kdv_exact}
    \begin{gathered}
        u_{\rm ex}(x,t)=A\,\operatorname{sech}^{2/3}\!\left(K(x-x_0)-\omega t\right),\\[-2pt] A=0.2275,\qquad K=3\left(\frac{A^3}{40\epsilon}\right)^{1/2},\qquad \omega=K\left(1+\frac{A^3}{10}\right),\qquad x_0=0.5.
    \end{gathered}
\end{equation}
The initial and boundary conditions are prescribed from the exact solution:
\begin{equation}\label{eq:kdv_ic_bc}
    \begin{aligned}
        u(x,0)&=u_{\rm ex}(x,0),\qquad u(-2,t)=u_{\rm ex}(-2,t),\\ u_x(3,t)&=\partial_xu_{\rm ex}(3,t),\qquad u_{xx}(3,t)=\partial_{xx}u_{\rm ex}(3,t).
    \end{aligned}
\end{equation}
This nonlinear example contains a moving localized wave structure. The reported solver uses a fixed-operator Picard iteration:
\begin{equation}\label{eq:kdv_fixed_picard}
    \mathcal L_{\rm fix}v=v_t+v_x+\epsilon v_{xxx}, \qquad \mathcal L_{\rm fix}u^{(k+1)} =-(u^{(k)})^3u_x^{(k)}.
\end{equation}
In the reported run, AD--RaNN updates the initial scale from $[10,10]$ to $[10.199,10.194]$, SFR reallocates the fixed 2000-feature budget as $m_S=995$ and $m_B=1005$, and the accepted residual-aligned RBE blocks are then appended using the same fixed operator.

\begin{figure}[!ht]
    \centering
    \includegraphics[width=0.32\textwidth]{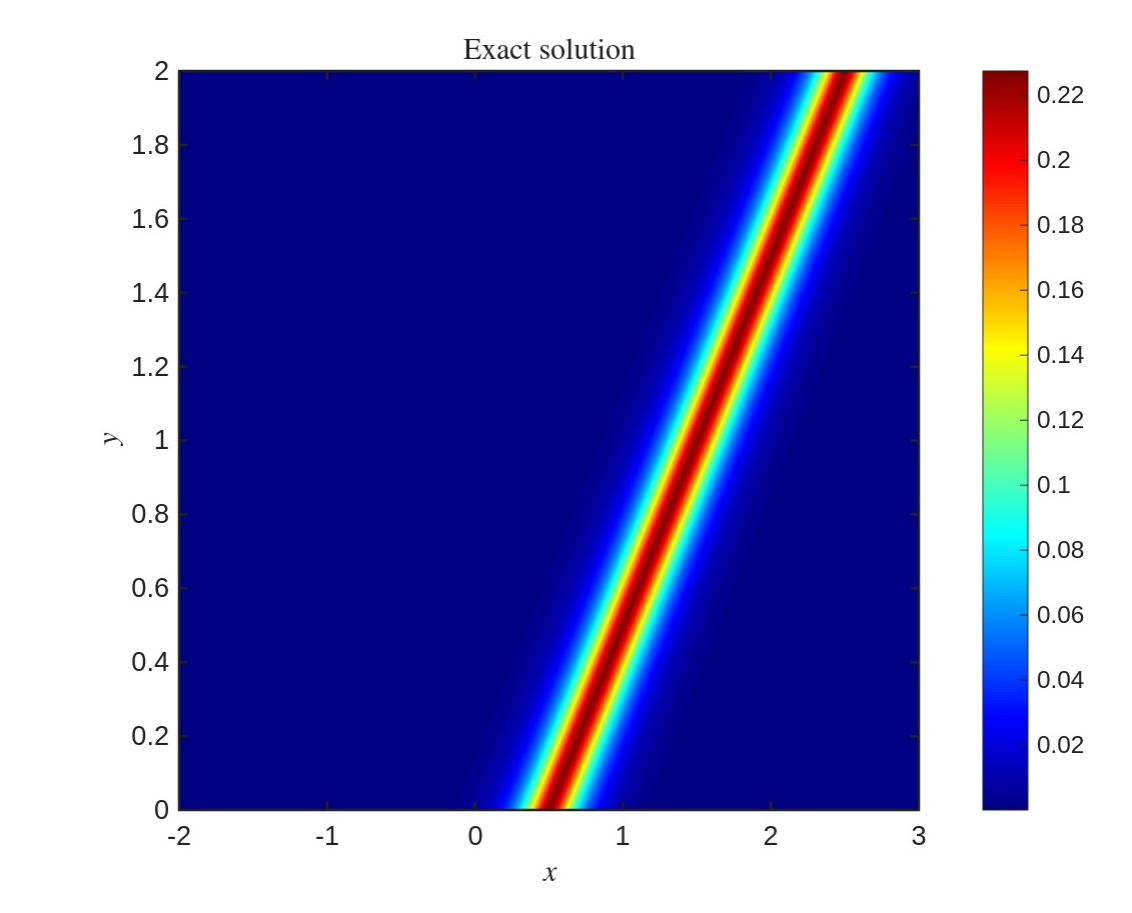}
    \includegraphics[width=0.32\textwidth]{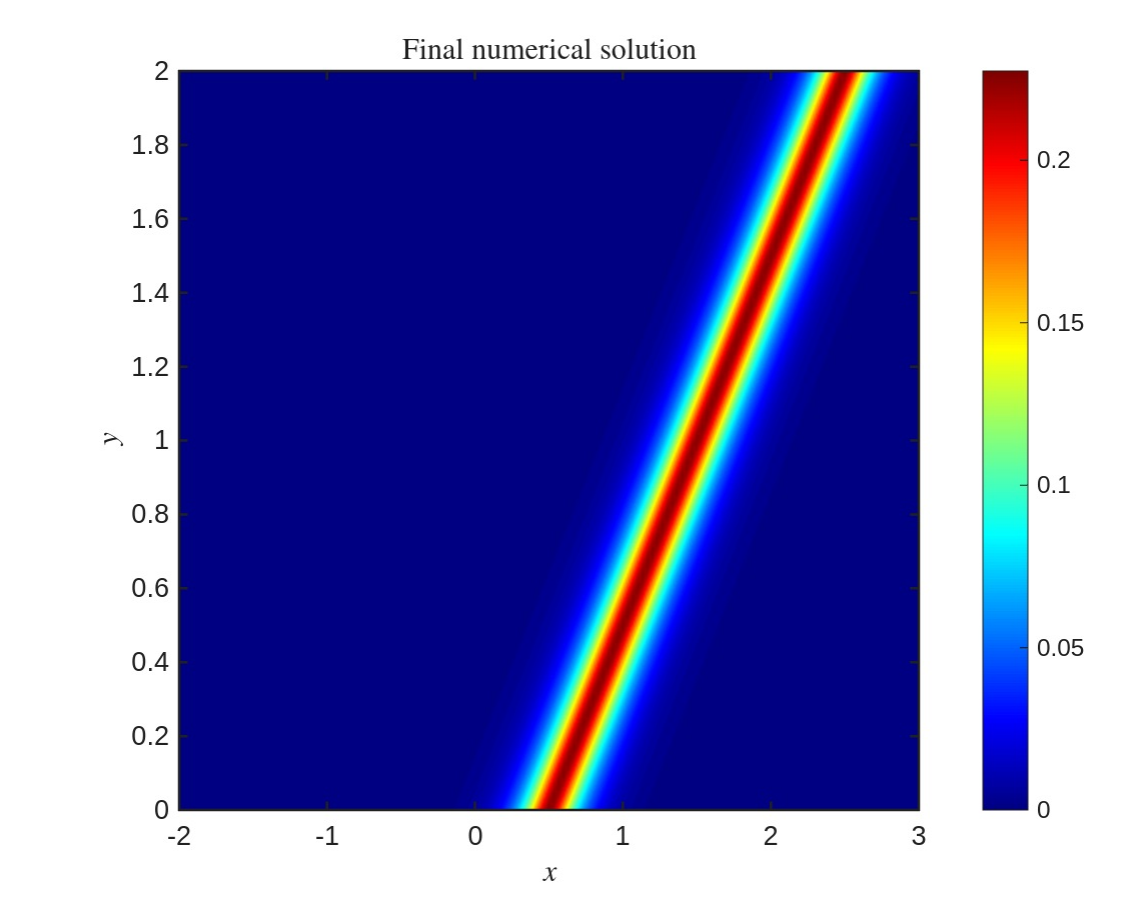}
    \includegraphics[width=0.32\textwidth]{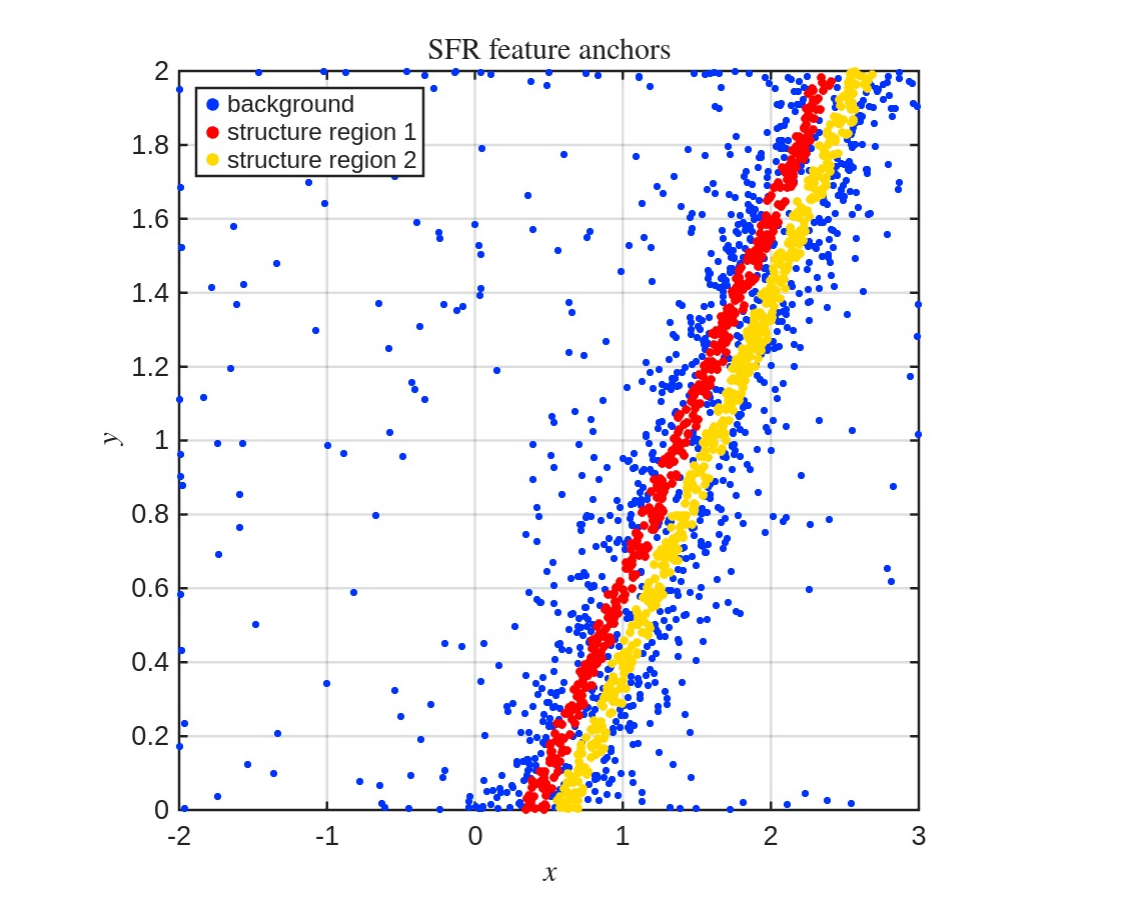}
    \caption{ Generalized KdV equation~\eqref{eq:kdv}. Exact solution (left), final numerical prediction after RBE enrichment (middle), and SFR feature-anchor distribution (right). The SFR feature anchors are concentrated around the moving localized wave structure. }
    \label{fig:kdv_solution_center}
\end{figure}

The fixed-operator initial solve has relative $L^2$ error $6.778\times10^{-3}$. SFR reduces it to $1.039\times10^{-5}$ while retaining 2000 features. Two accepted RBE blocks then add 400 residual-aligned features from \texttt{tanh} and \texttt{ridge\_gauss}, giving 2400 total features and a final relative $L^2$ error of $7.423\times10^{-7}$.

\begin{table}[!ht]
    \centering
    \caption{Stage-wise results for the fixed-operator generalized KdV equation~\eqref{eq:kdv}. The SFR rows report two structure components and the background, with $m_S$ denoting the total features allocated to the structure components. ``Total'' is the cumulative feature count. Times are stage wall-clock times; detailed RBE blocks are listed in Appendix~\ref{app:rbe_logs}.}
    \label{tab:kdv}
    \small
    \setlength{\tabcolsep}{3.5pt}
    \renewcommand{\arraystretch}{1.08}
    \resizebox{\linewidth}{!}{%
    \begin{tabular}{ccccccc}
        \toprule Stage & Family & Features & Total & Scale $\bm r$ & Time (s) & Rel. $L^2$\\
        \midrule Initial AD--RaNN & \texttt{tanh} & 2000 & 2000 & [10.199,10.194] & 34.005 & $6.778\mathrm{e}{-03}$\\
        \midrule \multirow{2}{*}{SFR} & \texttt{tanh} & $m_S=995$ & \multirow{2}{*}{2000} & \shortstack[l]{ $\mathcal S_1:[15.521,13.112]$\\
        $\mathcal S_2:[13.113,11.290]$} & \multirow{2}{*}{20.495} & \multirow{2}{*}{$1.039\mathrm{e}{-05}$}\\
        & \texttt{tanh} & $m_B=1005$ & & $[13.482,9.377]$ & & \\
        \midrule RBE & \texttt{tanh}/\texttt{ridge\_gauss} & 400 (2 blocks) & 2400 & {--} & 70.834 & $7.423\mathrm{e}{-07}$\\
        \bottomrule
    \end{tabular}}
\end{table}

\paragraph{Case 2: Periodic KdV two-soliton interaction}

We further consider the periodic KdV equation with double-soliton interaction:
\begin{equation}\label{eq:kdv2}
    u_t+u u_x+\epsilon u_{xxx}=0, \qquad (x,t)\in(0,2)\times(0,2),
\end{equation}
where $\epsilon=4.8400\times10^{-4}$. The periodic boundary conditions \(u(0,t)=u(2,t)\), \(u_x(0,t)=u_x(2,t)\), and \(u_{xx}(0,t)=u_{xx}(2,t)\) are imposed.
\begin{samepage}
    The two-pulse initial state is
    \begin{equation}\label{eq:kdv2_data}
        u(x,0)=3c_1\operatorname{sech}^2\!\left(k_1(x-x_1)\right) +3c_2\operatorname{sech}^2\!\left(k_2(x-x_2)\right),
    \end{equation}
    with \(c_1=0.3\), \(c_2=0.1\), \(x_1=0.4\), \(x_2=0.8\), \(\epsilon=4.84\times10^{-4}\), and \(k_i=\tfrac12\sqrt{c_i/\epsilon}\).
\end{samepage}

The equation solve uses damped Newton iterations. At the current iterate $u^{(k)}$,
\begin{equation}\label{eq:kdv2_newton}
    \begin{aligned}
        \mathcal J_k[\delta u] &=\delta u_t+u^{(k)}\delta u_x+u_x^{(k)}\delta u+\epsilon\,\delta u_{xxx},\\[-2pt] \delta\bm c^{(k)} &=\underset{\delta\bm c}{\arg\min}\, \|\bm J_k\delta\bm c+\bm R_k\|_2^2, \qquad \bm c^{(k+1)} =\bm c^{(k)}+\alpha_k\delta\bm c^{(k)} .
    \end{aligned}
\end{equation}
Here $\alpha_k\in\{1,2^{-1},2^{-2},\ldots\}$ is selected by backtracking to ensure a non-increasing nonlinear residual. The same refinement is applied after the initial solve, the SFR update, and each accepted RBE block.

Unlike the preceding benchmarks, this periodic KdV problem has no closed-form reference solution, so errors are measured against a high-accuracy Fourier spectral solution. The moving solitons and interaction region challenge globally sampled random features.
\begin{figure}[!ht]
    \centering
    \includegraphics[width=0.32\textwidth]{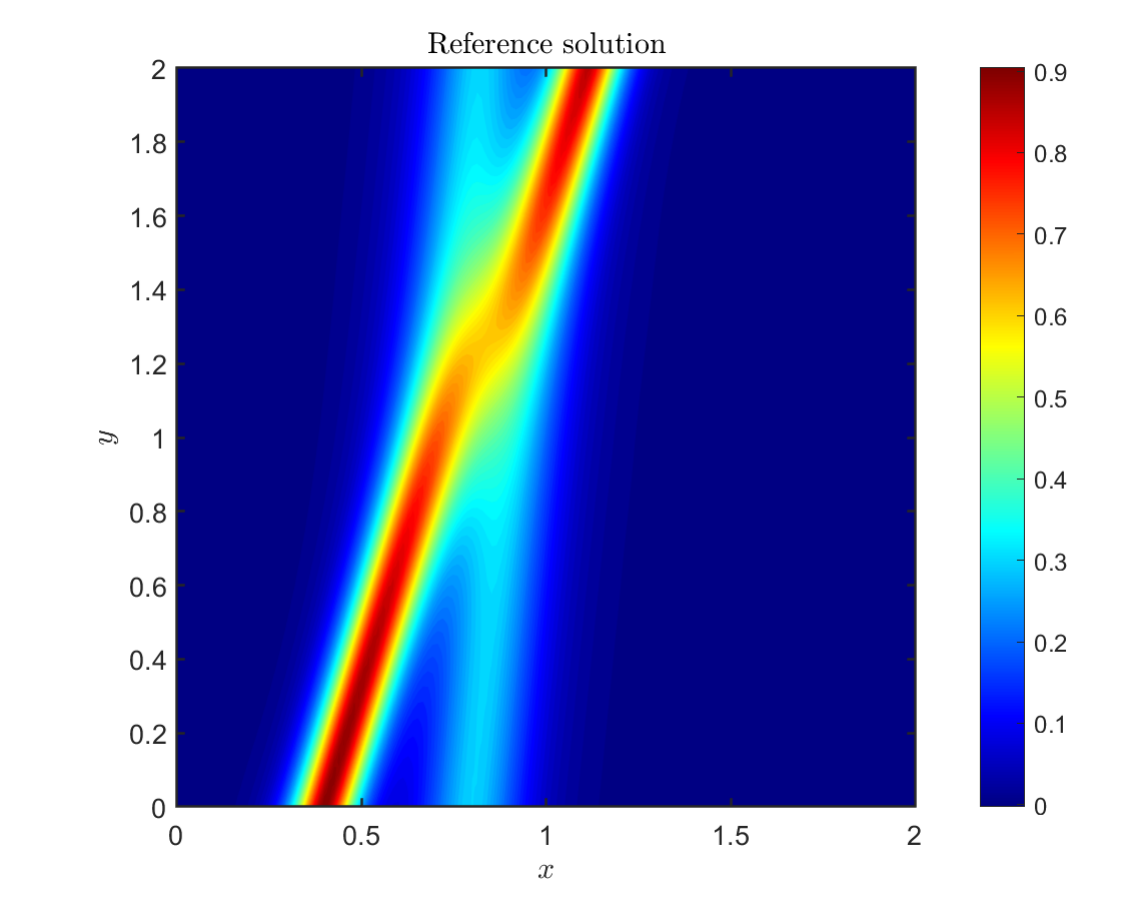}
    \includegraphics[width=0.32\textwidth]{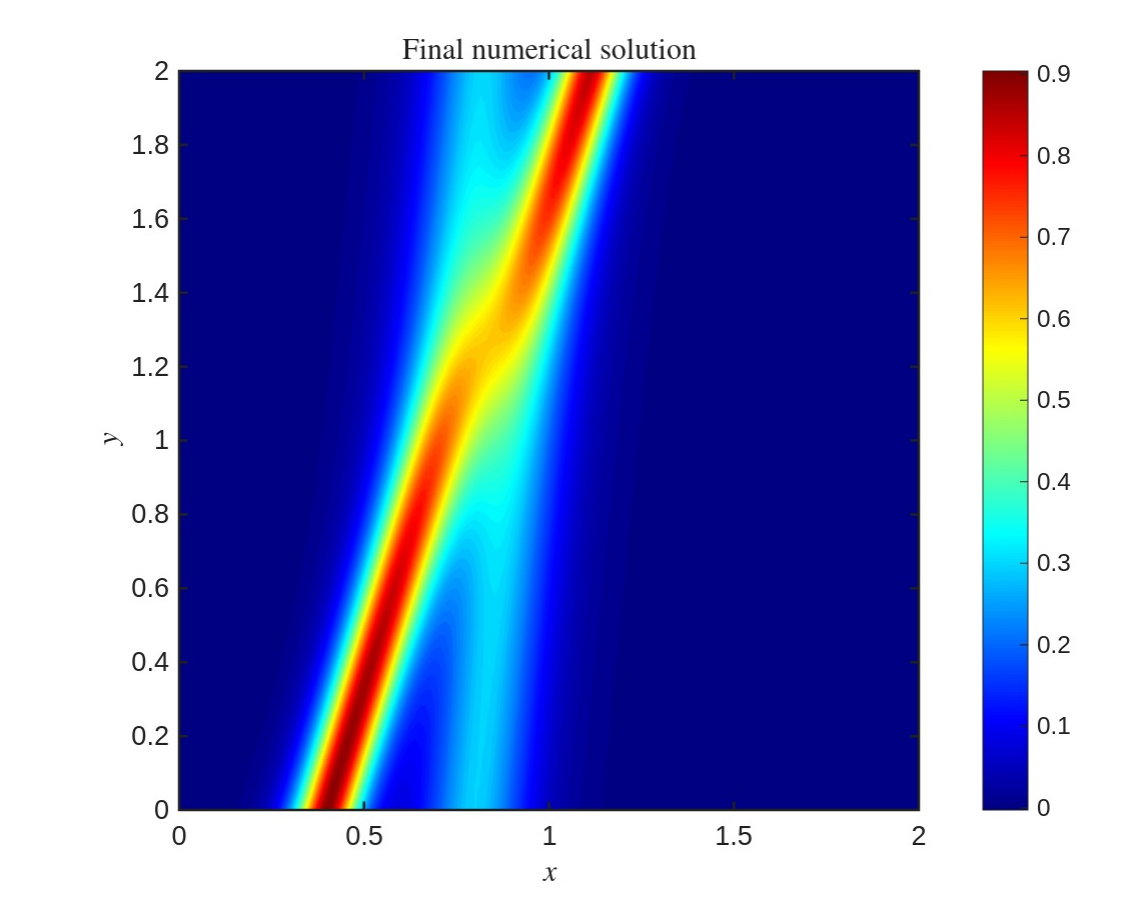}
    \includegraphics[width=0.32\textwidth]{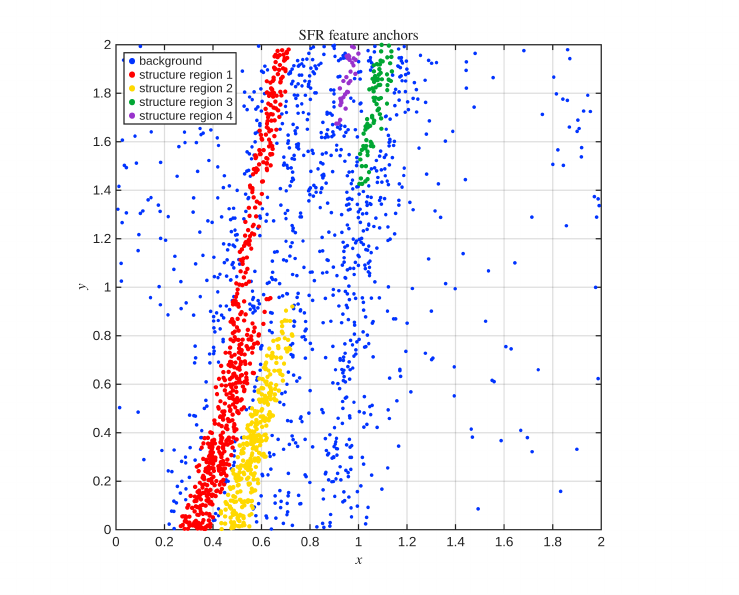}
    \caption{ Periodic KdV two-soliton interaction problem~\eqref{eq:kdv2}. Left: high-accuracy Fourier spectral reference solution. Middle: final numerical prediction after RBE enrichment. Right: SFR feature-anchor distribution. }
    \label{fig:kdv2_solution_center}
\end{figure}

To assess seed sensitivity, we repeat the same solver for seeds $0,\ldots,9$ without saving figures. The representative seed-4 run uses damped Newton refinement, changes the AD--RaNN scale from $[10,10]$ to $[14.8043,9.5169]$, allocates the fixed 2000-feature SFR budget as $m_S=984$ and $m_B=1016$, and accepts three RBE blocks. Across the ten seeds, the relative $L^2$ error decreases from $2.976\times10^{-1}$ (best $1.788\times10^{-1}$, worst $3.822\times10^{-1}$) initially to $4.225\times10^{-3}$ after SFR (best $2.219\times10^{-4}$, worst $1.435\times10^{-2}$) and to $9.061\times10^{-5}$ after RBE (best $6.648\times10^{-5}$, worst $1.262\times10^{-4}$). The accepted block count ranges from two to five (mean $3.7$), with a mean final width of $3130$ features. The representative seed-4 blocks and their scales are listed in Appendix~\ref{app:rbe_logs}.

\begin{table}[!ht]
    \centering
    \caption{Stage-wise numerical results for the periodic KdV two-soliton problem~\eqref{eq:kdv2} over seeds \(0,\ldots,9\). The SFR entry also reports the mean allocations to the structure and background regions; the RBE entry reports the mean cumulative feature count, with the mean number of added features and accepted blocks given in parentheses. Times are mean stage wall-clock times, and the values after $\pm$ denote one standard deviation over the ten seeds.}
    \label{tab:kdv2}
    \renewcommand{\arraystretch}{1.18}
    \setlength{\tabcolsep}{8pt}
    \resizebox{\linewidth}{!}{%
    \begin{tabular}{cccccc}
        \toprule Stage & Features (avg.) & Time (s, avg.) & Rel. \(L^2\) (avg.) & Rel. \(L^2\) (best) & Rel. \(L^2\) (worst)\\
        \midrule Initial AD--RaNN & \(2000\) & \(87.203\) & \(2.976\mathrm{e}{-01}\pm6.354\mathrm{e}{-02}\) & \(1.788\mathrm{e}{-01}\) & \(3.822\mathrm{e}{-01}\)\\
        \midrule SFR & \(2000\;(m_S=995.4,\ m_B=1004.6)\) & \(96.761\) & \(4.225\mathrm{e}{-03}\pm4.807\mathrm{e}{-03}\) & \(2.219\mathrm{e}{-04}\) & \(1.435\mathrm{e}{-02}\)\\
        \midrule RBE & \(3130\;(1130\ \text{added},\ 3.7\ \text{blocks})\) & \(579.581\) & \(9.061\mathrm{e}{-05}\pm2.333\mathrm{e}{-05}\) & \(6.648\mathrm{e}{-05}\) & \(1.262\mathrm{e}{-04}\)\\
        \bottomrule
    \end{tabular}}
\end{table}

\paragraph{Case 3: Burgers equation}

We finally consider the nonlinear Burgers equation on the three-dimensional space--time domain:
\begin{equation}\label{eq:burgers3d}
    u_t+u(u_x+u_y)-\epsilon(u_{xx}+u_{yy})=0, \qquad (x,y,t)\in(0,1)^2\times(0,1), \qquad \epsilon=3\times10^{-2}.
\end{equation}
The initial and boundary data are prescribed from $u_{\rm ex}(x,y,t)=\frac{1}{1+\exp\!\left((x+y-t)/(2\epsilon)\right)}$.

We use the stabilized fixed-operator iteration with the nonlinear convection term lagged from the previous iterate; the global marked set is used as one SFR structure mask.
\begin{figure}[!ht]
    \centering
    \includegraphics[width=0.24\textwidth]{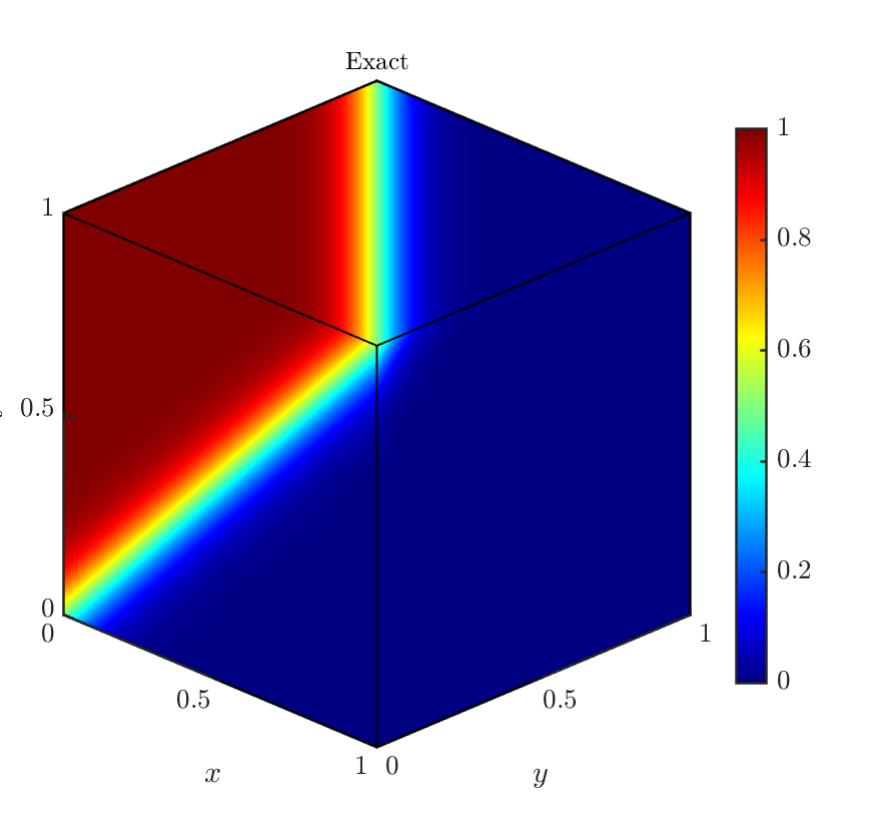}
    \includegraphics[width=0.24\textwidth]{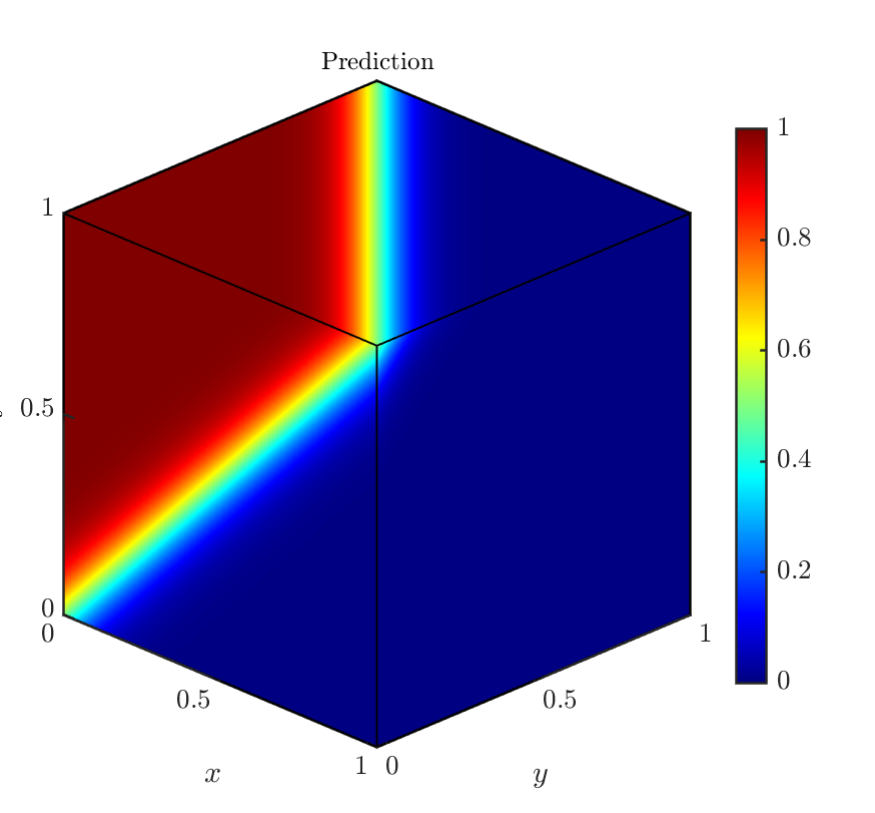}
    \includegraphics[width=0.24\textwidth]{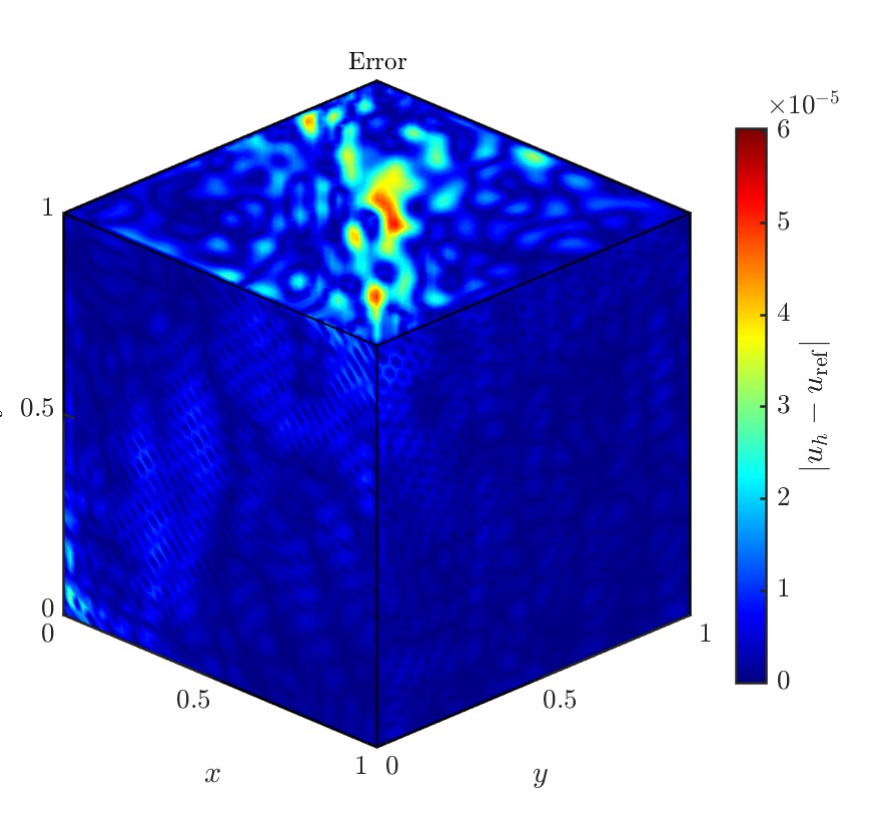}
    \includegraphics[width=0.24\textwidth]{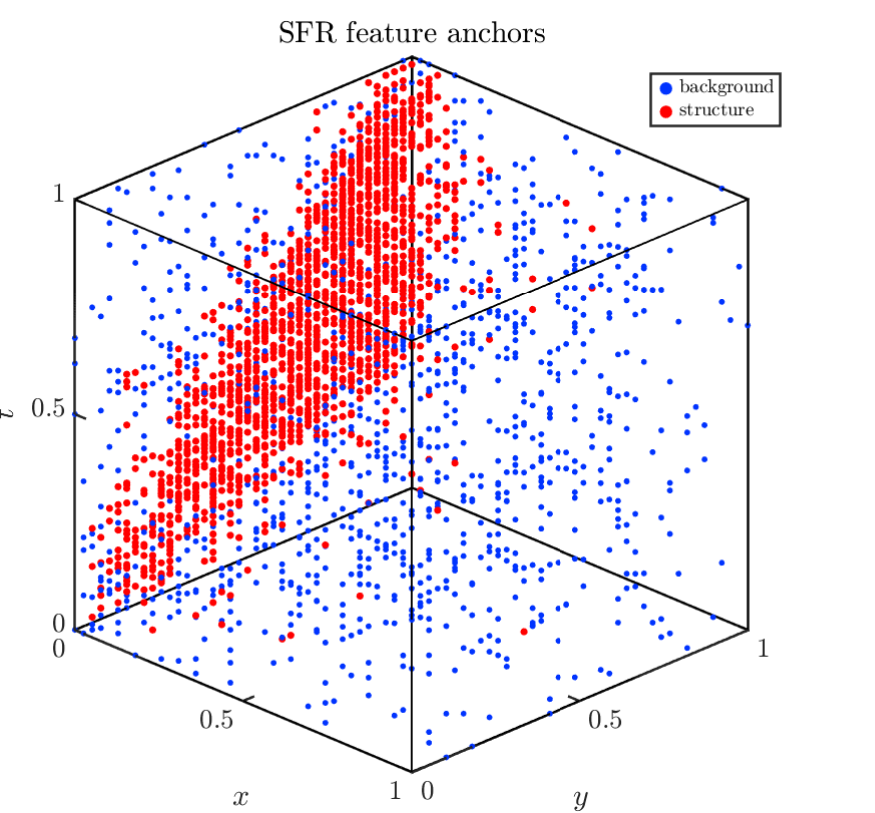}
    \caption{Space--time visualization for the Burgers problem \eqref{eq:burgers3d}: exact solution, final prediction after RBE enrichment, pointwise absolute error, and the SFR feature-anchor distribution.}
    \label{fig:burgers3d_summary}
\end{figure}
The relative $L^2$ error decreases from $8.202\times10^{-3}$ for the initial AD--RaNN space to $1.622\times10^{-4}$ after SFR and to $2.271\times10^{-5}$ after four accepted RBE blocks. The accepted blocks use the families \texttt{tanh}, \texttt{ridge\_gauss}, \texttt{tanh}, and \texttt{ridge\_gauss}, adding 1000 features in total.
\begin{table}[!ht]
    \centering
    \caption{Stage-wise results for the Burgers problem \eqref{eq:burgers3d}. The two SFR rows report the structure- and background-region neuron counts $m_S$ and $m_B$; ``Total'' is the cumulative feature count. Times are stage wall-clock values and exclude figure generation.}
    \label{tab:burgers3d}
    \small
    \setlength{\tabcolsep}{3.5pt}
    \renewcommand{\arraystretch}{1.08}
    \resizebox{\linewidth}{!}{%
    \begin{tabular}{ccccccc}
        \toprule Stage & Family & Features & Total & Scale $\boldsymbol r$ & Time (s) & Rel. $L^2$\\
        \midrule Initial AD--RaNN & \texttt{tanh} & 3000 & 3000 & [10.864,9.549,7.862] & 182.302 & $8.202\mathrm{e}{-03}$\\
        \midrule \multirow{2}{*}{SFR} & \texttt{tanh} & $m_S=1499$ & \multirow{2}{*}{3000} & \multirow{2}{*}{\shortstack{$[12.948,12.071,11.089]$\\$[8.500,8.986,8.639]$}} & \multirow{2}{*}{160.578} & \multirow{2}{*}{$1.622\mathrm{e}{-04}$}\\
        & \texttt{tanh} & $m_B=1501$ & & & & \\
        \midrule RBE & \texttt{tanh}/\texttt{ridge\_gauss} & 1000 (4 blocks) & 4000 & {--} & 564.769 & $2.271\mathrm{e}{-05}$\\
        \bottomrule
    \end{tabular}}
\end{table}

\section{Conclusion}

This paper introduced SFRBE--RaNN, an adaptive randomized neural PDE solver that preserves the output-linear least-squares structure while combining structure-aware feature redistribution with residual-driven block enrichment. SFR reallocates a fixed feature budget through adaptive anchors, regional neuron counts, and local scales, whereas RBE selects activation families and nonredundant feature blocks according to their residual-capture ability. Numerical experiments on localized, multiscale, variable-coefficient elliptic, and nonlinear space--time problems demonstrate the complementary effects of the two stages, while the analysis establishes exact fixed-linear residual reduction, a conditional finite-horizon RBE success-probability estimate and a safeguarded baseline-to-final error bound, together with a separate nonlinear acceptance estimate. Future work will develop efficient large-scale least-squares solvers for growing randomized feature systems, extend the framework to operator learning and parametric or multiphysics problems, and establish uniform candidate-coverage and finite-sample consistency bounds for fully data-dependent trial-space adaptation.

\bibliographystyle{plain}
\bibliography{references}

\appendix

\section{Proof of the Poisson residual-to-error estimate}
\label{app:poisson_residual_bound}

\begin{proof}[Proof of the Poisson stability estimate]
    Write \(v-u^*=w+z\), where \(-\Delta w=R(v)\), \(w|_{\partial\Omega}=0\), and \(z\) is harmonic with trace \(b(v)\). Since the first Dirichlet eigenvalue of the unit square is \(2\pi^2\),
    \[ \|w\|_{L^2(\Omega)} \le \frac{\|R(v)\|_{L^2(\Omega)}}{2\pi^2}. \]
    Let \(-\Delta\tau=1\) in \(\Omega\) with \(\tau|_{\partial\Omega}=0\). The maximum principle and comparison give \(0\le-\partial_n\tau\le1/2\) on \(\partial\Omega\). Green's identity then yields
    \[ \|z\|_{L^2(\Omega)}^2 \le -\int_{\partial\Omega}b(v)^2\partial_n\tau \le \frac12\|b(v)\|_{L^2(\partial\Omega)}^2, \]
    and hence \( \|z\|_{L^2(\Omega)} \le \frac{\|b(v)\|_{L^2(\partial\Omega)}}{\sqrt{2}}. \) The triangle inequality now gives \eqref{eq:poisson_l2_bound}.
\end{proof}

\section{Optional family-parameter refinement for SFRBE--RaNN}
\label{app:rbe_parameter_refinement}

The scale and family parameters in SFRBE--RaNN can be selected by several admissible mechanisms, including frequency-informed rules, bounded evolutionary searches such as differential evolution, and low-dimensional adaptive-distribution optimization~\cite{dang2024adaptive,dong2022Oncomputing,yang2026adaptivedistribution}. In this paper, both SFR and RBE use the AD--RaNN-style low-dimensional search, in which Adam updates the parameter vector $\boldsymbol\vartheta_F\in\Theta_F$ while the sampled anchors and directions remain fixed. Here $\Theta_F$ denotes the admissible parameter set for family $F$ used in the bounded parameter search. For a candidate family $F$, let
\begin{equation}
    C_{F,\perp}^{\rm scr}(\boldsymbol\vartheta_F)=(I-Q_LQ_L^{\mathsf T})C_F^{\rm scr}(\boldsymbol\vartheta_F). \label{eq:rbe_reduced_residualized_response}
\end{equation}
Starting from the initial value $\boldsymbol\vartheta_F^{(0)}$, Adam minimizes the ridge-regularized reduced objective
\begin{equation}
    \mathcal J_{F,L}^{\rm red,\lambda}(\boldsymbol\vartheta_F) = \min_{\bm a} \left\{ \left\|C_{F,\perp}^{\rm scr}(\boldsymbol\vartheta_F)\bm a-\bm r_L\right\|_2^2 + \lambda_{\rm scale}\|\bm a\|_2^2 \right\}, \qquad \boldsymbol\vartheta_F\in\Theta_F, \label{eq:rbe_reduced_tuning_objective}
\end{equation}
for at most $T_{\rm tune}$ bounded updates. The same regularization parameter $\lambda_{\rm scale}=10^{-6}$ is used in the reduced parameter searches for the initial AD--RaNN construction, SFR, and RBE. The ridge term is used only during this low-dimensional parameter search; after $\boldsymbol\vartheta_F$ is selected, the regularization is removed and the output coefficients are recomputed by the unregularized weighted least-squares solve.

\section{Additional implementation details}
\label{app:hyperparameters}

This appendix summarizes the implementation settings used in the numerical experiments. Unless otherwise stated, the output coefficients are computed by unregularized linear or linearized least-squares solves. The regularization parameter $\lambda_{\rm scale}=10^{-6}$ is used to stabilize the reduced parameter-search objectives in the initial AD--RaNN construction, SFR, and RBE. Tables~\ref{tab:app_problem_settings} and~\ref{tab:app_sfr_rbe_settings} summarize problem-dependent settings and shared controls, respectively.

\begin{table}[!ht]
    \centering
    \caption{Problem-dependent discretization, features, and constraint-weight settings. }
    \label{tab:app_problem_settings}
    \scriptsize
    \renewcommand{\arraystretch}{1.08}
    \setlength{\tabcolsep}{3pt}
    \begin{tabularx}{\textwidth}{
        @{} >{\raggedright\arraybackslash}p{0.23\textwidth} >{\centering\arraybackslash}X >{\centering\arraybackslash}p{0.14\textwidth} >{\centering\arraybackslash}p{0.13\textwidth} >{\centering\arraybackslash}p{0.20\textwidth} @{}} \toprule Problem & Points & Initial features & Full block size $m_{\rm full}$ & Weights $(\sqrt{\eta_b},\sqrt{\eta_0})$\\
        \midrule Two-peak Poisson & $N_{\rm pde}=22500,\ N_b=800$ & $3500$ & $1000$ & $(10^5,\text{--})$\\
        Multiscale Helmholtz & $N_{\rm pde}=62500,\ N_b=800$ & $3500$ & $2000$ & $(10^5,\text{--})$\\
        Multiscale Darcy & $N_{\rm pde}=22500,\ N_b=800$ & $3000$ & $1000$ & $(100,\text{--})$\\
        Generalized KdV & $N_{\rm pde}=22500,\ N_b=400,\ N_0=200$ & $2000$ & $1000$ & $(100,100)$\\
        Periodic KdV & $N_{\rm pde}=32400,\ N_b=400,\ N_0=200$ & $2000$ & $1000$ & $(100,100)$\\
        Burgers & $N_{\rm pde}=37026,\ N_b=4760,\ N_0=1225$ & $3000$ & $1000$ & $(300,300)$\\
        \bottomrule
    \end{tabularx}
\end{table}
\begin{table}[!ht]
    \centering
    \caption{Shared SFR/RBE and numerical implementation controls.}
    \label{tab:app_sfr_rbe_settings}
    \scriptsize\setlength{\tabcolsep}{4pt}
    \begin{tabular}{@{}p{0.35\textwidth}p{0.30\textwidth}p{0.26\textwidth}@{}}
        \toprule Module & Parameters & Values\\\midrule SFR & $\vartheta,\tau_{\rm eff},\tau_{\rm conn},\tau_{\rm cover}$ & $0.50,0.35,0.01,0.35$\\
        RBE screening & $\tau_{\rm svd},\tau_{\rm cap}$ & $10^{-10},10^{-4}$\\
        RBE screening size & $m_{\rm scr}$ & $\min\{500,\lceil0.3m_{\rm full}\rceil\}$\\
        RBE prefix/stop & $\theta_{\rm pref},\tau_{\rm stop}$ & $0.80,0.25$\\
        Stabilization & $\varepsilon_n$ & $10^{-12}$\\
        SFR parameter refinement & $\mathrm{lr}_S,T_S$ & $(0.7,30)$\\
        Family-parameter search & $\mathrm{lr}_F,T_{\rm tune}$ & $(0.5,30)$\\
        \bottomrule
    \end{tabular}
\end{table}
The prescribed prefix-size set is defined as
\[ \mathcal P =\left\{\left\lceil \alpha m_{\rm full}\right\rceil :\alpha\in\{0.2,0.4,0.5,0.6,0.7,0.8,0.9,1.0\} \right\}\]. The same sampling-cap rule is applied independently to the structure region, the background region, and the RBE candidate blocks. For low-dimensional parameter optimization, we subsample features and collocation equations separately for the SFR structure and background regions and RBE candidate blocks. Let $q\in\{S,B,R\}$, with $m_q$ selected features and $N_{{\rm eq},q}$ available equation rows. The sampling cap is
\begin{equation}
    N_q^{\max}=\min\left\{ N_{{\rm eq},q}, \left\lceil m_q\log(1+m_q)\right\rceil \right\},\qquad q\in\{S,B,R\}.
\end{equation}

\section{Stagewise activation-family and scale records}
\label{app:rbe_logs}
The following per-round records use representative seed-4 runs. Poisson and periodic KdV aggregate statistics remain in the main text.
\begin{table}[!ht]
    \centering
    \caption{Activation-family choices and residual decrease during RBE enrichment. Each row gives the post-acceptance weighted relative residual for the indicated block; the two panels are a compact continuation of one table.}
    \label{tab:app_RBE_details}
    \scriptsize
    \setlength{\tabcolsep}{2.2pt}
    \renewcommand{\arraystretch}{1.02}
    \begin{minipage}[t]{0.49\linewidth}
        \centering
        \resizebox{\linewidth}{!}{%
        \begin{tabular}{@{}r l r c r@{}}
            \toprule Round & Family & $m$ & Scale/shape & Rel. residual\\
            \midrule \multicolumn{5}{@{}l}{\textbf{Two-peak Poisson (Seed 4)}}\\
            1 & \texttt{radial\_gauss} & 200 & [47.872,50.505] & $6.001\mathrm{e}{-04}$\\
            2 & \texttt{mexican\_hat} & 200 & [55.822,50.104] & $1.992\mathrm{e}{-04}$\\
            3 & \texttt{ridge\_gauss} & 400 & [69.134,66.557] & $9.468\mathrm{e}{-05}$\\
            4 & \texttt{mexican\_hat} & 500 & [48.317,57.077] & $5.354\mathrm{e}{-05}$\\
            5 & \texttt{sin} & 600 & [48.802,50.577] & $3.093\mathrm{e}{-05}$\\
            6 & \texttt{mexican\_hat} & 500 & [46.264,71.383] & $1.796\mathrm{e}{-05}$\\
            7 & \texttt{sin} & 500 & [51.141,52.627] & $7.823\mathrm{e}{-06}$\\
            8 & \texttt{ridge\_gauss} & 400 & [76.559,73.034] & $4.993\mathrm{e}{-06}$\\
            \addlinespace[1pt] \multicolumn{5}{@{}l}{\textbf{Multiscale Helmholtz}}\\
            1 & \texttt{sin} & 400 & [61.490,67.019] & $2.507\mathrm{e}{-01}$\\
            2 & \texttt{tanh} & 400 & [67.066,70.204] & $1.381\mathrm{e}{-01}$\\
            3 & \texttt{sin} & 400 & [61.165,67.295] & $6.122\mathrm{e}{-02}$\\
            4 & \texttt{ridge\_gauss} & 400 & [80.840,84.406] & $3.355\mathrm{e}{-02}$\\
            5 & \texttt{sin} & 400 & [61.015,66.926] & $9.133\mathrm{e}{-03}$\\
            6 & \texttt{ridge\_gauss} & 400 & [82.788,86.523] & $4.532\mathrm{e}{-03}$\\
            7 & \texttt{sin} & 400 & [61.121,68.547] & $7.007\mathrm{e}{-04}$\\
            8 & \texttt{ridge\_gauss} & 800 & [82.173,88.652] & $3.576\mathrm{e}{-04}$\\
            9 & \texttt{sin} & 800 & [60.006,68.498] & $1.328\mathrm{e}{-04}$\\
            10 & \texttt{mexican\_hat} & 1200 & [5.239,23.696] & $1.072\mathrm{e}{-04}$\\
            \bottomrule
        \end{tabular}}
    \end{minipage}\hfill
    \begin{minipage}[t]{0.49\linewidth}
        \centering
        \resizebox{\linewidth}{!}{%
        \begin{tabular}{@{}r l r c r@{}}
            \toprule Round & Family & $m$ & Scale/shape & Rel. residual\\
            \midrule \multicolumn{5}{@{}l}{\textbf{Multiscale Darcy}}\\
            1 & \texttt{mexican\_hat} & 500 & [35.224,36.122] & $9.792\mathrm{e}{-05}$\\
            2 & \texttt{ridge\_gauss} & 700 & [35.261,36.979] & $5.996\mathrm{e}{-05}$\\
            3 & \texttt{tanh} & 500 & [45.750,44.683] & $3.699\mathrm{e}{-05}$\\
            4 & \texttt{ridge\_gauss} & 700 & [35.228,35.666] & $2.608\mathrm{e}{-05}$\\
            5 & \texttt{tanh} & 800 & [48.606,43.850] & $1.976\mathrm{e}{-05}$\\
            6 & \texttt{ridge\_gauss} & 900 & [33.089,34.127] & $1.587\mathrm{e}{-05}$\\
            7 & \texttt{tanh} & 800 & [53.972,48.430] & $1.307\mathrm{e}{-05}$\\
            \addlinespace[1pt] \multicolumn{5}{@{}l}{\textbf{Generalized KdV (fixed operator)}}\\
            1 & \texttt{tanh} & 200 & [15.718,13.432] & $3.964\mathrm{e}{-07}$\\
            2 & \texttt{ridge\_gauss} & 200 & [28.334,20.426] & $1.676\mathrm{e}{-07}$\\
            \addlinespace[1pt] \multicolumn{5}{@{}l}{\textbf{Periodic KdV two-soliton (Seed 4)}}\\
            1 & \texttt{ridge\_gauss} & 200 & [27.148,14.723] & $5.518\mathrm{e}{-05}$\\
            2 & \texttt{tanh} & 200 & [33.628,14.985] & $4.000\mathrm{e}{-05}$\\
            3 & \texttt{sin} & 600 & [23.463,15.269] & $3.486\mathrm{e}{-05}$\\
            \addlinespace[1pt] \multicolumn{5}{@{}l}{\textbf{Burgers}}\\
            1 & \texttt{tanh} & 200 & [13.118,15.097,12.596] & $3.540\mathrm{e}{-05}$\\
            2 & \texttt{ridge\_gauss} & 400 & [18.520,20.163,12.438] & $2.714\mathrm{e}{-05}$\\
            3 & \texttt{tanh} & 200 & [20.770,24.351,15.088] & $1.963\mathrm{e}{-05}$\\
            4 & \texttt{ridge\_gauss} & 200 & [25.983,32.015,19.770] & $1.572\mathrm{e}{-05}$\\
            \addlinespace[1pt] \bottomrule
        \end{tabular}}
    \end{minipage}
\end{table}

\section*{Code availability}
The repository address for the MATLAB implementation, experiment configurations, and plotting scripts is \url{https://github.com/XJTU-AI4SciComp-Lab/sfrbe-rann-adaptive-paper-code}.

\end{document}